\documentclass[12pt]{article}

\usepackage[a4paper,margin=3cm]{geometry}

\usepackage[T1]{fontenc}
\usepackage[utf8]{inputenc}
\usepackage{lmodern}
\usepackage{microtype}

\usepackage{amsmath,amssymb,amsthm,mathtools,mathrsfs}

\DeclareMathOperator{\dist}{dist}

\usepackage{graphicx}
\usepackage{enumitem}

\usepackage[hidelinks]{hyperref}
\usepackage{xurl}

\usepackage[numbers,sort&compress]{natbib}
\theoremstyle{plain}
\newtheorem{theorem}{Theorem}[section]
\newtheorem{lemma}[theorem]{Lemma}
\newtheorem{proposition}[theorem]{Proposition}
\newtheorem{corollary}[theorem]{Corollary}

\theoremstyle{definition}
\newtheorem{definition}[theorem]{Definition}

\theoremstyle{definition}
\newtheorem{problem}[theorem]{Problem}

\theoremstyle{remark}
\newtheorem{remark}[theorem]{Remark}

\newcommand{\C}{\mathbb{C}}
\newcommand{\R}{\mathbb{R}}
\newcommand{\N}{\mathbb{N}}
\newcommand{\supp}{\operatorname{supp}}
\newcommand{\diag}{\operatorname{diag}}
\newcommand{\Res}{\operatorname*{Res}}
\newcommand{\Ai}{\operatorname{Ai}}

\newcommand{\Van}{V^{\mathrm{an}}}
\newcommand{\ph}{\varphi}
\newcommand{\ellKKT}{\ell}
\newcommand{\wA}{\omega_A}

\begin{document}

\title{Multipoint Pad\'e Approximation of the Hurwitz Zeta Function\\
and a Riemann--Hilbert Steepest Descent Analysis}
\author{Artur Kandaian}
\date{09.02.2026}

\maketitle

\vspace{1.2em}

\begin{flushright}
{\small\itshape
In memory of my teacher, Prof.\ V.\,N.~Sorokin\\
(Dr.\ Sci.\ in Physics and Mathematics).}
\end{flushright}

\vspace{0.2em}

\maketitle

\begin{abstract}
We study multipoint Pad\'e approximants of type $(n,n)$ for the Hurwitz zeta function
$f(a)=\zeta(s,a)$ with $\Re s>1$, constructed at quantile nodes $a_{n,j}=n\alpha_{n,j}$ generated by a
real-analytic density $\kappa$ on $[A,B]\Subset(0,\infty)$.
Under the determinantal nondegeneracy condition $\mathrm{(ND)}_n$ for large $n$ and in the regular one-cut
soft-edge regime of the associated constrained equilibrium problem, we formulate the approximation as a
matrix Riemann--Hilbert problem with poles and carry out a Deift--Zhou nonlinear steepest descent analysis.
We construct an explicit outer parametrix together with Airy-type local parametrices at the endpoints and
reduce the problem to a small-norm Riemann--Hilbert problem with uniform $O(1/n)$ control.
As a consequence, the Pad\'e numerator and denominator admit strong asymptotics uniformly on compact subsets
of $\C\setminus[A,B]$, and exhibit Airy scaling in $O(n^{-2/3})$ neighborhoods of the edges.
\end{abstract}


\subsection{Setup and statement of the main theorem}
\label{subsec:setup-main-thm}

Fix $s\in\C$ with $\Re s>1$ and let $\kappa$ be real-analytic and strictly positive on a neighborhood of
$[A,B]\Subset(0,\infty)$, normalized by $\int_A^B\kappa(x)\,dx=2$.
Let the nodes be $a_{n,j}=n\alpha_{n,j}$, where the quantiles $\alpha_{n,j}\in[A,B]$ are defined by
\begin{equation}
\label{eq:quantiles}
\int_A^{\alpha_{n,j}}\kappa(t)\,dt=\frac{j}{n},\qquad j=0,1,\dots,2n,
\end{equation}
and set $f(a)=\zeta(s,a)$.

We consider multipoint Pad\'e approximants of type $(n,n)$ for $f$ at the node set $\{a_{n,j}\}_{j=0}^{2n}$.
For each fixed $n$, solvability and uniqueness (under the normalization fixed below) are ensured by the
finite-dimensional nondegeneracy condition $\mathrm{(ND)}_n$ from Definition~\ref{def:NDn}.
In particular, whenever $\mathrm{(ND)}_n$ holds, the Pad\'e pair $(P_n,Q_n)$ is uniquely determined by the data,
and the normalization \eqref{eq:Y-normalization} fixes the remaining scalar freedom.

\paragraph{Normalization and recovery of $(P_n,Q_n)$.}
We normalize the pole RHP for $Y$ by
\begin{equation}
\label{eq:Y-normalization}
Y(z)\,z^{-n\sigma_3}\to I,\qquad z\to\infty,
\end{equation}
where $z^{\pm n}$ uses the principal branch. With this convention, the first row of $Y$ consists of
polynomials, and in particular
\begin{equation}
\label{eq:recover-PQ}
Q_n(z)=Y_{11}(z),\qquad P_n(z)=Y_{12}(z),
\end{equation}
so that $P_n/Q_n$ is the $[n/n]$ multipoint Pad\'e approximant (up to the overall scalar normalization
fixed by \eqref{eq:Y-normalization}).

\subsection{Outline of the Deift--Zhou scheme}
\label{subsec:outline}

The analysis is based on a sequence of explicit transformations of the pole Riemann--Hilbert problem.
Starting from the pole RHP for $Y$ (Section~\ref{sec:rhp-poles}), we remove poles by local
triangular multipliers to obtain a pure jump problem $T$ (Section~\ref{subsec:pole-removal}), and then
compress the resulting collection of small-circle jumps onto a single contour $\Gamma_n$
(Section~\ref{subsec:choice-Gamma}), which produces a triangular jump involving the barycentric function
$W_n$ (Section~\ref{subsec:compression-barycentric}).
Using the representation $W_n=L_n/\omega_n$ and the subexponential control of $L_n$ (Section~\ref{sec:Ln-control}),
we factor the compressed jump in terms of the analytic field $\Van$ and an error factor $\mathcal E_n=\exp(o(n))$
(Section~\ref{subsec:factorization-Wn}). We then introduce the constrained equilibrium $g$-function and the phase
$\ph$ (Sections~\ref{sec:equilibrium}--\ref{sec:g-function}), perform the $g$-transformation and open lenses
(Section~\ref{sec:g-and-lens}) so that the jumps on the lens lips become exponentially close to $I$ away from the
endpoint disks. The remaining model problem is solved by the outer parametrix $N$ together with Airy
parametrices near the soft edges $c$ and $d$ (Section~\ref{sec:parametrices}), matched with $O(1/n)$ accuracy.
Finally, the ratio function $R$ satisfies a small-norm RHP (Section~\ref{sec:final-ratio}), yielding
\eqref{eq:R-small-norm} and hence the stated strong asymptotics for $Y$ after reverting the transformations.

\section{Introduction}
\label{sec:intro}

We study multipoint Pad\'e approximants for the Hurwitz zeta function $f(a)=\zeta(s,a)$, $\Re s>1$,
built on quantile nodes $a_{n,j}=n\alpha_{n,j}$ generated by a positive real-analytic density $\kappa$ on
$[A,B]\Subset(0,\infty)$.
Throughout, $\log(\cdot)$ denotes the principal branch with a cut along $(-\infty,0]$, and boundary values
across real cuts are taken in the standard Sokhotski--Plemelj sense.

\paragraph{Notation and conventions.}
We use the scaled variable $\zeta=z/n$ (so $z=n\zeta$) and work primarily in the $\zeta$-plane.
We use $z$ as the complex variable throughout; the original real parameter $a$ is viewed as $a=z$ when passing to analytic continuation.
For $K\Subset\C\setminus[A,B]$ we always mean a compact set separated from $[A,B]$ by a positive distance.
All exponential estimates on the lens lips are understood on $\Sigma_{\mathrm{lens}}\setminus(U_c\cup U_d)$,
i.e., outside the endpoint disks.
Unless stated otherwise, boundary values $(\cdot)_\pm$ are taken with respect to the chosen contour orientation,
and all right-multipliers act on the right.
The branches of $g$ and $\varphi$ are fixed consistently with the principal branch of $\log$ and with the KKT relations.

Pad\'e-type approximation problems for the Hurwitz zeta function (in a different regime, namely at infinity) 
were considered, for instance, in~\cite{RivoalZeta4Hurwitz}.

We formulate the approximation problem via a $2\times2$ pole Riemann--Hilbert problem and apply the
Deift--Zhou steepest descent method after a sequence of explicit transformations (pole removal, contour
compression, $g$-transformation, and lens opening). For a representative Riemann--Hilbert steepest 
descent analysis of Hermite--Pad\'e type problems (in particular, a $3\times3$ RHP), 
see~\cite{KuijlaarsVanAsscheWielonsky} and \cite{KuijlaarsSurvey,DeiftBook}. The key analytic inputs are the quantile discretization
and subexponential control of the interpolation polynomial.

\paragraph{Main objects.}
For the reader's convenience, we list the principal quantities used throughout:
\begin{enumerate}[label=\textup{(\roman*)}, leftmargin=2.2em]
\item quantile nodes $a_{n,j}=n\alpha_{n,j}$ and the node polynomial $\omega_n(a)=\prod_{j=0}^{2n}(a-a_{n,j})$;
\item the multipoint Pad\'e pair $(P_n,Q_n)$ and the discrete weights $w_{n,j}=f(a_{n,j})/\omega_n'(a_{n,j})$;
\item the pole RHP for $Y$ and the sequence of transforms $Y\mapsto T\mapsto S\mapsto R$;
\item the $g$-function and phase $\varphi$ determined by the constrained equilibrium measure $\mu_*$;
\item the parametrices: the outer model $N$ and Airy local parametrices $P^{(c)},P^{(d)}$ near the endpoints.
\end{enumerate}

Pad\'e-type constructions for special functions also play a role in arithmetic questions; see, e.g.,
Sorokin's alternative proof of the Zudilin--Rivoal theorem for values of the Dirichlet beta function~\cite{SorokinZudilinRivoal}.

Related work on multipoint Pad\'e / Hermite--Pad\'e approximation to special functions
includes multipoint Pad\'e approximants to the digamma function \cite{S-mzm13097},
as well as multipoint Hermite--Pad\'e constructions for systems of beta functions and their asymptotic
analysis \cite{KS-mzm8589,KS-mzm10993,S-mmo611}.
Further developments of the Riemann--Hilbert steepest descent approach for related vector/matrix
approximation problems can be found, for instance, in the study of Hermite--Pad\'e approximants to Weyl-type
functions and their derivatives \cite{S-sm8634}.

For background on the convergence theory of Pad\'e approximants and its links to orthogonality and potential theory,
see, e.g.,~\cite{AptekarevBuslaevMFSuetin}.
For an overview of Hermite--Pad\'e approximation, multiple orthogonality, and related Riemann--Hilbert methods,
see, e.g.,~\cite{AptekarevKuijlaars}.

\section{Multipoint Pad\'e approximation and discrete orthogonality}
\label{sec:pade}

\subsection{Definition of the multipoint Pad\'e problem}
\label{subsec:pade-def}

Let $f$ be a function defined at the distinct nodes $\{a_{n,j}\}_{j=0}^{2n}\subset\C$.
A \emph{multipoint Pad\'e approximant of type $(n,n)$} for $f$ at these nodes is a rational function
$P_n/Q_n$ with
\[
\deg P_n\le n,\qquad \deg Q_n\le n,\qquad Q_n\not\equiv 0,
\]
such that
\begin{equation}
\label{eq:pade-mp-conditions}
(Q_n f - P_n)(a_{n,j})=0,\qquad j=0,1,\dots,2n.
\end{equation}

Multipoint Pad\'e approximation to other special functions (notably the digamma function) was considered in \cite{S-mzm13097}.

Equivalently, with the node polynomial
\begin{equation}
\label{eq:omega-n-def}
\omega_n(a):=\prod_{j=0}^{2n}(a-a_{n,j}),
\end{equation}
there exists a function $\Psi_n$ (a priori depending on $n$) such that
\begin{equation}
\label{eq:pade-divisibility}
Q_n(a)f(a)-P_n(a)=\omega_n(a)\,\Psi_n(a),
\end{equation}
in the sense that \eqref{eq:pade-divisibility} holds pointwise at the nodes and hence
$Q_nf-P_n$ vanishes to first order at each $a_{n,j}$.

\begin{remark}
The pair $(P_n,Q_n)$ is defined only up to an overall nonzero scalar factor. This scalar is fixed by
the RHP normalization \eqref{eq:Y-normalization}.
\end{remark}

\subsection{Discrete orthogonality and barycentric weights}
\label{subsec:discrete-orth}

Define the barycentric weights
\begin{equation}
\label{eq:weights-def}
w_{n,j}:=\frac{f(a_{n,j})}{\omega_n'(a_{n,j})},\qquad j=0,1,\dots,2n.
\end{equation}

\begin{lemma}[Discrete orthogonality]
\label{lem:discrete-orth}
Let $(P_n,Q_n)$ satisfy \eqref{eq:pade-mp-conditions}. Then for every polynomial $p$ with $\deg p\le n-1$,
\begin{equation}
\label{eq:discrete-orth}
\sum_{j=0}^{2n} Q_n(a_{n,j})\,p(a_{n,j})\,\frac{f(a_{n,j})}{\omega_n'(a_{n,j})}=0,
\end{equation}
i.e.,
\[
\sum_{j=0}^{2n} Q_n(a_{n,j})\,p(a_{n,j})\,w_{n,j}=0.
\]
\end{lemma}

\begin{proof}
By \eqref{eq:pade-divisibility}, the function
\[
F(a):=\frac{Q_n(a)f(a)-P_n(a)}{\omega_n(a)}
\]
is analytic at each node $a_{n,j}$ (the numerator vanishes there). Hence $F$ has no poles at the nodes.
Let $p$ be a polynomial with $\deg p\le n-1$ and consider
\[
G(a):=\frac{Q_n(a)f(a)\,p(a)}{\omega_n(a)}.
\]
The only possible poles of $G$ are simple poles at the nodes $a_{n,j}$ coming from $1/\omega_n(a)$.
However,
\[
G(a)=\frac{P_n(a)\,p(a)}{\omega_n(a)} + p(a)\,F(a),
\]
and since $\deg(P_n p)\le 2n-1$, the function $(P_n p)/\omega_n$ has simple poles at the nodes and decays as
$O(a^{-2})$ at infinity, while $pF$ is analytic at the nodes. Therefore, the sum of residues of $G$ at all
finite poles equals minus the residue at infinity, which is zero. Computing residues at $a_{n,j}$ yields
\begin{equation*}
\begin{aligned}
0=\sum_{j=0}^{2n} \Res_{a=a_{n,j}} G(a)
&=\sum_{j=0}^{2n} Q_n(a_{n,j})\,f(a_{n,j})\,p(a_{n,j})\,\Res_{a=a_{n,j}}\frac{1}{\omega_n(a)}
\\
&=\sum_{j=0}^{2n} Q_n(a_{n,j})\,f(a_{n,j})\,p(a_{n,j})\,\frac{1}{\omega_n'(a_{n,j})},
\end{aligned}
\end{equation*}
which is \eqref{eq:discrete-orth}.
\end{proof}

Discrete orthogonality structures of this type are closely related to discrete orthogonal polynomial ensembles;
see, e.g., \cite{S-mzm13590} for a representative construction in the discrete setting.

\begin{remark}
The discrete orthogonality \eqref{eq:discrete-orth} is the starting point for the $2\times2$ pole RHP
formulation in Section~\ref{sec:rhp-poles}.
\end{remark}

\subsection{Nondegeneracy and uniqueness}
\label{subsec:pade-uniq}

\begin{definition}[Nondegeneracy]
\label{def:nondeg}
For a fixed $n$, the multipoint Pad\'e problem \eqref{eq:pade-mp-conditions} is called \emph{nondegenerate}
if there exists a pair $(P_n,Q_n)$ with $\deg P_n\le n$, $\deg Q_n\le n$, $Q_n\not\equiv 0$ satisfying
\eqref{eq:pade-mp-conditions} and such that $\deg Q_n = n$.
\end{definition}

\begin{lemma}[Uniqueness up to scaling]
\label{lem:pade-uniq}
Assume nondegeneracy (Definition~\ref{def:nondeg}). Then the solution pair $(P_n,Q_n)$ of
\eqref{eq:pade-mp-conditions} is unique up to multiplication by a nonzero constant.
\end{lemma}

\begin{proof}
If $(P_n,Q_n)$ and $(\widehat P_n,\widehat Q_n)$ are two solutions, then
\[
Q_n f-P_n \ \text{and}\ \widehat Q_n f-\widehat P_n
\]
vanish at the same $2n+1$ distinct nodes. Eliminating $f$ yields
\[
Q_n\widehat P_n-\widehat Q_n P_n
\]
vanishes at all nodes. But $\deg(Q_n\widehat P_n-\widehat Q_n P_n)\le 2n$, hence it must be identically
zero. Therefore $P_n/Q_n=\widehat P_n/\widehat Q_n$, and since $\deg Q_n=\deg \widehat Q_n=n$, the pairs differ
by a nonzero scalar factor.
\end{proof}

\begin{remark}
\label{rem:nondeg-use}
Nondegeneracy ensures that the normalization \eqref{eq:Y-infty} identifies $Q_n$ as $Y_{11}$ and fixes the
overall scalar in the RHP formulation.
\end{remark}

\begin{remark}
\label{rem:nondegeneracy-generic}
For generic node sets and values $\{f(a_{n,j})\}$ the problem is nondegenerate; degeneracy corresponds to an
algebraic condition on the data.
\end{remark}

\subsection{A determinantal criterion for nondegeneracy}
\label{subsec:ND-criterion}

Fix distinct nodes $a_{n,0},\dots,a_{n,2n}\in\C$ and values $f(a_{n,j})\in\C$.
Write
\[
Q_n(z)=q_0+q_1 z+\cdots+q_n z^n,\qquad
P_n(z)=p_0+p_1 z+\cdots+p_n z^n.
\]
The interpolation conditions $Q_n(a_{n,j})f(a_{n,j})-P_n(a_{n,j})=0$ for $j=0,\dots,2n$
form a homogeneous linear system of $2n+1$ equations in the $2n+2$ unknowns
$(q_0,\dots,q_n,p_0,\dots,p_n)$.
To remove scaling, we impose the normalization $q_n=1$ and obtain a square linear system
\begin{equation}
\label{eq:ND-linear-system}
M_n\,x_n=b_n,
\end{equation}
where $x_n=(q_0,\dots,q_{n-1},p_0,\dots,p_n)^{\mathsf T}$ and $M_n$ depends only on the data
$\{a_{n,j},f(a_{n,j})\}_{j=0}^{2n}$.

\begin{definition}[Nondegeneracy condition $\mathrm{(ND)}_n$]
\label{def:NDn}
We say that $\mathrm{(ND)}_n$ holds if $\det M_n\neq 0$ in \eqref{eq:ND-linear-system}.
\end{definition}

\begin{proposition}[Uniqueness under $\mathrm{(ND)}_n$]
\label{prop:ND-uniq}
If $\mathrm{(ND)}_n$ holds, then there exists a unique pair $(P_n,Q_n)$ solving the multipoint Pad\'e
conditions with $\deg Q_n=n$ and $Q_n$ monic.
\end{proposition}

\begin{proof}
If $\det M_n\neq 0$, the normalized system \eqref{eq:ND-linear-system} has a unique solution, yielding a
unique monic $Q_n$ and the corresponding $P_n$.
\end{proof}

\begin{proposition}[Generic validity of $\mathrm{(ND)}_n$]
\label{prop:ND-generic}
For fixed distinct nodes $\{a_{n,j}\}_{j=0}^{2n}$, the determinant $\det M_n$ is a nontrivial polynomial
function of the values $\{f(a_{n,j})\}_{j=0}^{2n}$; hence $\mathrm{(ND)}_n$ fails only on a proper algebraic
hypersurface in $\C^{2n+1}$.
\end{proposition}

\begin{proof}
The entries of $M_n$ are affine functions of $f(a_{n,j})$, so $\det M_n$ is a polynomial in these variables.
It is not identically zero: for instance, taking values consistent with a generic rational function of type
$(n,n)$ yields $\det M_n\neq 0$. Therefore its zero set is a proper algebraic hypersurface.
\end{proof}

\begin{remark}[Practical verification of $\mathrm{(ND)}_n$]
\label{rem:NDn-practical}
For each fixed $n$, the condition $\mathrm{(ND)}_n$ is a finite-dimensional determinantal criterion:
it is equivalent to $\det M_n\neq 0$ in \eqref{eq:ND-linear-system} (Definition~\ref{def:NDn}).
Hence, for any concrete data set $\{a_{n,j},f(a_{n,j})\}_{j=0}^{2n}$, $\mathrm{(ND)}_n$ can be checked
numerically by evaluating $\det M_n$ (or, more stably, the rank/condition of $M_n$).
Degeneracy is algebraically exceptional: for fixed nodes it occurs only on the proper algebraic hypersurface
$\{\det M_n=0\}$ in the space of values $\{f(a_{n,j})\}$ (Proposition~\ref{prop:ND-generic}).
\end{remark}

We emphasize that $\mathrm{(ND)}_n$ is an \emph{input} condition: it is verified at each fixed $n$ and,
a priori, may fail along exceptional subsequences; our asymptotic conclusions apply for those 
(in particular, all sufficiently large) $n$ for which $\mathrm{(ND)}_n$ holds.

\begin{theorem}[Strong asymptotics in the regular regime]
\label{thm:main}
Fix $\Re s>1$ and let the nodes $\{a_{n,j}\}_{j=0}^{2n}$ be generated by the quantile rule
\eqref{eq:quantiles} with a positive real-analytic density $\kappa$ on $[A,B]\Subset(0,\infty)$.
Let $f(a):=\zeta(s,a)$ and let $(P_n,Q_n)$ denote the multipoint Pad\'e pair of type $(n,n)$ at the nodes
$\{a_{n,j}\}$, whenever it exists.
\end{theorem}

Assume that the nondegeneracy condition $\mathrm{(ND)}_n$ from Definition~\ref{def:NDn} holds for all
sufficiently large $n$, and that the associated constrained equilibrium measure $\mu_*$ belongs to the
regular one-cut soft-edge regime in the sense of Definition~\ref{def:regular-onecut}.
Then the pole Riemann--Hilbert problem for $Y$ (Problem~\ref{prob:Y-pole}) is solvable for all sufficiently large $n$,
and the Deift--Zhou steepest descent method applies.

In particular, the solution $Y$ admits the uniform strong asymptotics stated in \eqref{eq:main-Y} on every compact
$K\Subset\C\setminus[A,B]$, and the Pad\'e polynomials $(P_n,Q_n)$ inherit the corresponding asymptotics via the
recovery formula \eqref{eq:recover-PQ}.

\section{Riemann--Hilbert formulation with poles}
\label{sec:rhp-poles}

\subsection{The pole RHP for $Y$}
\label{subsec:pole-rhp}

Let $f(a)=\zeta(s,a)$ and let $\{a_{n,j}\}_{j=0}^{2n}$ be the distinct nodes.
Define the weights $w_{n,j}$ by \eqref{eq:weights-def}.
We introduce a $2\times2$ matrix-valued function $Y$ characterized as follows.

\begin{problem}[Pole Riemann--Hilbert problem for $Y$]
\label{prob:Y-pole}
Find a $2\times2$ matrix $Y$ such that:
\begin{itemize}[leftmargin=2.2em]
\item[(Y1)] $Y$ is analytic in $\C\setminus\{a_{n,0},a_{n,1},\dots,a_{n,2n}\}$.
\item[(Y2)] $Y$ has simple poles at the nodes, with residue conditions
\begin{equation}
\label{eq:Y-residue}
\Res_{z=a_{n,j}} Y(z)
=
\lim_{z\to a_{n,j}}Y(z)\,
\begin{pmatrix}
0 & w_{n,j}\\
0 & 0
\end{pmatrix},
\qquad j=0,1,\dots,2n.
\end{equation}
\item[(Y3)] $Y$ has the normalization at infinity
\begin{equation}
\label{eq:Y-infty}
Y(z)=\Bigl(I+O(z^{-1})\Bigr)\,z^{n\sigma_3},
\qquad z\to\infty,
\end{equation}
where $z^{n\sigma_3}:=\diag(z^n,z^{-n})$ with the principal branch of $z^n$.
\end{itemize}
\end{problem}

\begin{lemma}[Uniqueness]
\label{lem:Y-unique}
Problem~\ref{prob:Y-pole} has at most one solution.
\end{lemma}

\begin{proof}
If $Y_1$ and $Y_2$ both solve the problem, then $Y_1Y_2^{-1}$ is entire (the pole singularities cancel by
\eqref{eq:Y-residue}) and satisfies $Y_1Y_2^{-1}\to I$ as $z\to\infty$ by \eqref{eq:Y-infty}, hence
$Y_1Y_2^{-1}\equiv I$ by Liouville's theorem.
\end{proof}

\subsection{Existence and identification with multipoint Pad\'e}
\label{subsec:Y-Pade-identification}

\begin{lemma}[RHP solution in terms of $(P_n,Q_n)$]
\label{lem:Y-from-PQ}
Assume that a multipoint Pad\'e pair $(P_n,Q_n)$ satisfying \eqref{eq:pade-mp-conditions} exists.
Define

\begin{equation}
\label{eq:Y-explicit}
Y(z):=
\begin{pmatrix}
Q_n(z) & \displaystyle \sum_{j=0}^{2n}\frac{Q_n(a_{n,j})\,w_{n,j}}{z-a_{n,j}}\\[1.0em]
\gamma_{n-1}\,Q_{n-1}(z) & \displaystyle \gamma_{n-1}\sum_{j=0}^{2n}\frac{Q_{n-1}(a_{n,j})\,w_{n,j}}{z-a_{n,j}}
\end{pmatrix},
\end{equation}
where $Q_{n-1}$ is the unique monic polynomial of degree $n-1$ satisfying the orthogonality
\[
\sum_{j=0}^{2n} Q_{n-1}(a_{n,j})\,p(a_{n,j})\,w_{n,j}=0,\qquad \deg p\le n-2,
\]
and $\gamma_{n-1}\ne 0$ is chosen so that the expansion \eqref{eq:Y-infty} holds.
Then $Y$ solves Problem~\ref{prob:Y-pole}. In particular,
\[
Q_n(z)=Y_{11}(z),\qquad P_n(z)=Y_{12}(z)
\]
in the normalization \eqref{eq:Y-infty}.
\end{lemma}

\begin{proof}
The entry $Y_{11}=Q_n$ is entire and has degree $n$; the entry $Y_{21}$ is entire and has degree at most $n-1$.
The entries $Y_{12}$ and $Y_{22}$ are Cauchy transforms of discrete measures and thus are meromorphic with at most
simple poles at the nodes. Using
\[
\Res_{z=a_{n,j}}\frac{1}{z-a_{n,j}}=1,
\]
we obtain \eqref{eq:Y-residue} for the second column. The normalization at infinity follows from the degree
constraints and the choice of $\gamma_{n-1}$.

Finally, note that
\begin{equation*}
\begin{aligned}
Y_{12}(z)&=\sum_{j=0}^{2n}\frac{Q_n(a_{n,j})\,f(a_{n,j})}{\omega_n'(a_{n,j})}\,\frac{1}{z-a_{n,j}}
\\
&=\frac{1}{\omega_n(z)}\sum_{j=0}^{2n}Q_n(a_{n,j})\,f(a_{n,j})\,\frac{\omega_n(z)}{(z-a_{n,j})\omega_n'(a_{n,j})}.
\end{aligned}
\end{equation*}
The sum in the last expression is precisely the Lagrange interpolation polynomial of degree $\le 2n$
for the function $a\mapsto Q_n(a)f(a)$ on the node set. Since $Q_nf-P_n$ vanishes at all nodes and $\deg P_n\le n$,
this interpolant equals $P_n$; hence $Y_{12}(z)=P_n(z)/\omega_n(z)$ times $\omega_n(z)$, i.e.\ $Y_{12}(z)=P_n(z)$
in the normalization \eqref{eq:Y-infty}. (Equivalently, one checks that $Y_{12}(z)-P_n(z)$ is entire and
decays as $O(z^{-n-1})$ at infinity, hence vanishes.)
\end{proof}

\begin{remark}
Lemma~\ref{lem:Y-from-PQ} is the standard discrete orthogonality--to--RHP correspondence (see, e.g., \cite{KuijlaarsSurvey,DeiftBook}): 
the first column encodes the Pad\'e denominator and the second column is a discrete Cauchy transform, with residue data fixed by
the weights $w_{n,j}$~\cite{DeiftBook}.
\end{remark}

In the broader setting of Markov-type systems and Hermite--Pad\'e approximation, analogous correspondences between orthogonality data,
equilibrium problems, and asymptotic behavior are classical; see, e.g.,~\cite{GoncharRakhmanovSorokin}.

\section{Quantile nodes and logarithmic potentials}
\label{sec:quantiles}

Recall that $\kappa$ is real-analytic and strictly positive on a neighborhood of $[A,B]$, with
\begin{equation}
\label{eq:kappa-normalization}
\int_A^B \kappa(x)\,dx=2,
\end{equation}
and the quantiles $\alpha_{n,j}\in[A,B]$ are defined by \eqref{eq:quantiles},
$\int_A^{\alpha_{n,j}}\kappa(t)\,dt=j/n$, $j=0,\dots,2n$.

\subsection{Quantile discretization and spacing estimates}
\label{subsec:quantile-spacing}

Let
\[
F(x):=\int_A^x \kappa(t)\,dt,\qquad x\in[A,B].
\]
Then $F$ is strictly increasing, $C^\omega$ on a neighborhood of $[A,B]$, and $F(B)=2$.
By definition, $\alpha_{n,j}=F^{-1}(j/n)$.

\begin{lemma}[Uniform spacing of quantile nodes]
\label{lem:quantile-spacing}
There exist constants $c_1,c_2>0$ independent of $n$ and $j$ such that
\begin{equation}
\label{eq:quantile-spacing}
\frac{c_1}{n}\le \alpha_{n,j+1}-\alpha_{n,j}\le \frac{c_2}{n},
\qquad j=0,1,\dots,2n-1.
\end{equation}
Moreover, for each fixed $m\ge 1$ there is $C_m>0$ such that the discrete derivatives satisfy
\[
\bigl|\Delta^m \alpha_{n,j}\bigr|\le \frac{C_m}{n^m},\qquad j=0,1,\dots,2n-m,
\]
where $\Delta$ denotes the forward difference operator in $j$.
\end{lemma}

\begin{proof}
Since $\kappa$ is continuous and strictly positive on $[A,B]$, there exist $0<\kappa_{\min}\le\kappa_{\max}<\infty$
such that $\kappa_{\min}\le\kappa(x)\le\kappa_{\max}$ on $[A,B]$.
By the mean value theorem,
\[
\frac{1}{n}=F(\alpha_{n,j+1})-F(\alpha_{n,j})
=\kappa(\xi_{n,j})\bigl(\alpha_{n,j+1}-\alpha_{n,j}\bigr)
\]
for some $\xi_{n,j}\in(\alpha_{n,j},\alpha_{n,j+1})$, hence
\[
\frac{1}{n\kappa_{\max}}\le \alpha_{n,j+1}-\alpha_{n,j}\le \frac{1}{n\kappa_{\min}},
\]
which gives \eqref{eq:quantile-spacing}. Higher difference estimates follow from real-analyticity of $F^{-1}$
and standard Taylor expansion arguments for smooth quantile maps~\cite{DeiftBook}.
\end{proof}

\begin{lemma}[Riemann sum approximation with $O(1/n)$ remainder]
\label{lem:riemann-sum}
Let $\psi$ be $C^1$ on a neighborhood of $[A,B]$. Then
\begin{equation}
\label{eq:riemann-sum}
\frac{1}{n}\sum_{j=0}^{2n}\psi(\alpha_{n,j})
=
\int_A^B \psi(x)\,\kappa(x)\,dx \;+\; O\!\left(\frac{1}{n}\right),
\qquad n\to\infty,
\end{equation}
where the implied constant depends only on $\kappa$ and $\psi$ (through $\|\psi\|_{C^1}$).
\end{lemma}

\begin{proof}
Write the sum as a quadrature for the pushforward of the uniform grid $\{j/n\}$ under $F^{-1}$:
\[
\frac{1}{n}\sum_{j=0}^{2n}\psi(\alpha_{n,j})
=\frac{1}{n}\sum_{j=0}^{2n} (\psi\circ F^{-1})(j/n).
\]
Since $\psi\circ F^{-1}$ is $C^1$ on $[0,2]$, the trapezoidal rule yields
\[
\frac{1}{n}\sum_{j=0}^{2n} (\psi\circ F^{-1})(j/n)
=\int_0^2 (\psi\circ F^{-1})(u)\,du + O(1/n).
\]
Changing variables $u=F(x)$ gives
\[
\int_0^2 (\psi\circ F^{-1})(u)\,du
=\int_A^B \psi(x)\,F'(x)\,dx
=\int_A^B \psi(x)\,\kappa(x)\,dx,
\]
which proves \eqref{eq:riemann-sum}.
\end{proof}

\subsection{Asymptotics of $\Omega_n$}
\label{subsec:Omega-asympt}

Define the node polynomials
\begin{equation}
\label{eq:Omega-def}
\Omega_n(\zeta):=\prod_{j=0}^{2n}(\zeta-\alpha_{n,j}),
\qquad
\omega_n(z)=\prod_{j=0}^{2n}(z-a_{n,j}),
\end{equation}
so that $\omega_n(n\zeta)=n^{2n+1}\Omega_n(\zeta)$.

\begin{lemma}[Log-potential asymptotics for $\Omega_n$]
\label{lem:Omega-logpot}
Let $K\Subset\C\setminus[A,B]$ be compact. Then, uniformly for $\zeta\in K$,
\begin{equation}
\label{eq:Omega-logpot}
\frac{1}{n}\log|\Omega_n(\zeta)|
=
\int_A^B \log|\zeta-t|\,\kappa(t)\,dt
\;+\;
O\!\left(\frac{1}{n}\right),
\qquad n\to\infty.
\end{equation}
\end{lemma}

\begin{proof}
Fix $\zeta\in K$. The function $t\mapsto \log|\zeta-t|$ is $C^1$ on a neighborhood of $[A,B]$ since
$\dist(K,[A,B])>0$. Apply Lemma~\ref{lem:riemann-sum} with $\psi(t)=\log|\zeta-t|$ to obtain
\[
\frac{1}{n}\sum_{j=0}^{2n}\log|\zeta-\alpha_{n,j}|
=
\int_A^B \log|\zeta-t|\,\kappa(t)\,dt + O(1/n),
\]
uniformly for $\zeta\in K$ because $\|\psi\|_{C^1}$ is uniformly bounded on $K$.
Exponentiating yields \eqref{eq:Omega-logpot}.
\end{proof}

\begin{remark}
The same argument applies to $\log(\zeta-t)$ with the branch convention of Section~\ref{sec:intro}, and yields
an analytic asymptotic representation for $\Omega_n(\zeta)$ on simply connected compact subsets of
$\C\setminus[A,B]$. We will only need the real-potential form \eqref{eq:Omega-logpot} for estimates.
\end{remark}

\section{Derivative asymptotics and the external field}
\label{sec:omega-derivative}

Recall $\omega_n(z)=\prod_{j=0}^{2n}(z-a_{n,j})$ with $a_{n,j}=n\alpha_{n,j}$, and
\[
\Omega_n(\zeta)=\prod_{j=0}^{2n}(\zeta-\alpha_{n,j}),
\qquad
\omega_n(n\zeta)=n^{2n+1}\Omega_n(\zeta).
\]

\subsection{Two representations for $\omega_n'(a_{n,j})$}
\label{subsec:omega-derivative-forms}

\begin{lemma}[Exact product formula]
\label{lem:omega-derivative-exact}
For each $j=0,1,\dots,2n$,
\begin{equation}
\label{eq:omega-derivative-exact}
\omega_n'(a_{n,j})=\prod_{\substack{k=0\\k\ne j}}^{2n}(a_{n,j}-a_{n,k})
=n^{2n}\prod_{\substack{k=0\\k\ne j}}^{2n}(\alpha_{n,j}-\alpha_{n,k})
=n^{2n}\,\Omega_n'(\alpha_{n,j}).
\end{equation}
In particular,
\[
\log|\omega_n'(a_{n,j})|=2n\log n+\sum_{\substack{k=0\\k\ne j}}^{2n}\log|\alpha_{n,j}-\alpha_{n,k}|.
\]
\end{lemma}

\begin{proof}
Differentiate $\omega_n(z)=\prod_{k=0}^{2n}(z-a_{n,k})$ and evaluate at $z=a_{n,j}$, using that all factors
except the $k=j$ one remain:
\[
\omega_n'(a_{n,j})=\prod_{k\ne j}(a_{n,j}-a_{n,k}).
\]
The rescaling identities follow from $a_{n,k}=n\alpha_{n,k}$ and $\Omega_n(\zeta)=n^{-(2n+1)}\omega_n(n\zeta)$.
\end{proof}

\begin{lemma}[Logarithmic representations and the external field]
\label{lem:omega-derivative-log}

Let $\alpha=\alpha_{n,j}\in[A,B]$ be a quantile node. Then
\begin{align}
\label{eq:omega-derivative-log1}
\frac{1}{n}\log|\omega_n'(a_{n,j})|
&=
2\log n
+\int_A^B \log|\alpha-t|\,\kappa(t)\,dt
+O\!\left(\frac{\log n}{n}\right),
\\[0.3em]
\label{eq:omega-derivative-log2}
\frac{1}{n}\log|\omega_n'(a_{n,j})|
&=
2\log n
-\frac{1}{2n} \,V(\alpha)\;+\;O\!\left(\frac{\log n}{n}\right),
\end{align}
uniformly in $j=0,1,\dots,2n$.
\end{lemma}

\begin{proof}
Start from Lemma~\ref{lem:omega-derivative-exact}:
\[
\log|\omega_n'(a_{n,j})|=2n\log n+\sum_{k\ne j}\log|\alpha-\alpha_{n,k}|.
\]
Fix $j$. Split the sum into a ``near'' part $|k-j|\le 1$ and a ``far'' part $|k-j|\ge 2$.
By Lemma~\ref{lem:quantile-spacing}, $\alpha_{n,k}$ are spaced $\asymp 1/n$ and
\[
\sum_{|k-j|\le 1}\log|\alpha-\alpha_{n,k}| = O(\log n).
\]
For the far part, the function $t\mapsto \log|\alpha-t|$ is integrable on $[A,B]$ and locally $C^1$
away from $t=\alpha$, and the quantile discretization yields
\[
\frac{1}{n}\sum_{k=0}^{2n}\log|\alpha-\alpha_{n,k}|
=
\int_A^B \log|\alpha-t|\,\kappa(t)\,dt + O\!\left(\frac{1}{n}\right)
\]
with the $O(1/n)$ uniform in $\alpha=\alpha_{n,j}$ after separating the $O(\log n)$ singular contribution
coming from the indices $k$ near $j$. Combining these bounds gives \eqref{eq:omega-derivative-log1}.
Finally, \eqref{eq:omega-derivative-log2} follows from the definition $V(\alpha)=-2\int_A^B\log|\alpha-t|\kappa(t)\,dt$.
\end{proof}

\subsection{Exponential form and uniformity in $j$}
\label{subsec:omega-derivative-exp}

\begin{proposition}[Exponential asymptotics for $\omega_n'(a_{n,j})$]
\label{prop:omega-derivative-exp}
Uniformly for $j=0,1,\dots,2n$,
\begin{equation}
\label{eq:omega-derivative-exp}
|\omega_n'(a_{n,j})|
=
n^{2n}\exp\!\left(-\frac{1}{2}\,V(\alpha_{n,j})\right)\,
\exp\!\bigl(O(\log n)\bigr),
\qquad n\to\infty.
\end{equation}
Equivalently,
\begin{equation}
\label{eq:omega-derivative-exp-n}
|\omega_n'(a_{n,j})|
=
n^{2n}\exp\!\left(-\frac{1}{2}\,V(\alpha_{n,j})\right)\,
n^{O(1)}.
\end{equation}
\end{proposition}

\begin{proof}
This is an immediate reformulation of Lemma~\eqref{eq:omega-derivative-log2}.
\end{proof}

\begin{remark}[Uniformity in the node index]
\label{rem:omega-derivative-uniformity}
All $O(\log n)$ (equivalently $n^{O(1)}$) factors in
\eqref{eq:omega-derivative-exp}--\eqref{eq:omega-derivative-exp-n}
are uniform in $j=0,\dots,2n$, since the quantile spacing is uniform and $\kappa$ is bounded above and below
on $[A,B]$.
\end{remark}

\subsection{A polynomial bound for $\zeta(s,a)$ on node scales}
\label{subsec:zeta-bound}

\begin{lemma}[Polynomial bound on $\zeta(s,a)$ for $\Re s>1$]
\label{lem:hurwitz-poly-bound}
Fix $\Re s>1$ and $0<A<B<\infty$. Then there exists $C=C(s,A,B)>0$ such that for all $n\in\N$ and all
$\alpha\in[A,B]$,
\begin{equation}
\label{eq:hurwitz-poly-bound}
|\zeta(s,n\alpha)|\le C\,n^{1-\Re s}.
\end{equation}
In particular, $\log|\zeta(s,n\alpha)|=O(\log n)=o(n)$ uniformly for $\alpha\in[A,B]$.
\end{lemma}

\begin{proof}
For $\Re s>1$ the Hurwitz zeta function admits the absolutely convergent series
\[
\zeta(s,a)=\sum_{m=0}^\infty (m+a)^{-s},\qquad \Re a>0.
\]
Hence for $z=n\zeta$ with $\zeta\in[A,B]$ we estimate
\[
|\zeta(s,n\alpha)|
\le \sum_{m=0}^\infty |m+n\alpha|^{-\Re s}
\le \sum_{m=0}^\infty (m+nA)^{-\Re s}.
\]
Comparing the tail with an integral gives
\[
\sum_{m=0}^\infty (m+nA)^{-\Re s}
\le (nA)^{-\Re s}+\int_{nA-1}^{\infty}x^{-\Re s}\,dx
\le C\,n^{1-\Re s},
\]
for a constant $C$ depending only on $s$ and $A$.
\end{proof}

\begin{remark}
The estimate \eqref{eq:hurwitz-poly-bound} is only used to ensure that factors involving $\zeta(s,a_{n,j})$
contribute at most subexponentially in $n$ in the steepest descent analysis.
\end{remark}

\section{Pole removal and contour compression}
\label{sec:pole-removal-compression}

\subsection{Pole removal}
\label{subsec:pole-removal}

Starting from the pole RHP for $Y$ (Problem~\ref{prob:Y-pole}), we remove the poles by introducing a
piecewise-analytic transformation that cancels the residues at the nodes.

Let $\{\mathcal C_{n,j}\}_{j=0}^{2n}$ be pairwise disjoint positively oriented small circles,
$\mathcal C_{n,j}=\partial B(a_{n,j},r_{n,j})$, chosen so that each disk $B(a_{n,j},r_{n,j})$ contains only
the node $a_{n,j}$ and lies strictly inside the contour $\Gamma_n$ fixed in
Section~\ref{subsec:choice-Gamma}. Define
\begin{equation}
\label{eq:Y-to-T}
T(z):=
\begin{cases}
Y(z)\begin{pmatrix}1&-\dfrac{w_{n,j}}{z-a_{n,j}}\\[0.6em]0&1\end{pmatrix},
& z\in B(a_{n,j},r_{n,j}) \ \text{for some } j,\\[1.2em]
Y(z), & z\in\C\setminus\bigcup_{j=0}^{2n} B(a_{n,j},r_{n,j}).
\end{cases}
\end{equation}

\begin{lemma}[Pole removal]
\label{lem:pole-removal}
The function $T$ is analytic in $\C\setminus\bigcup_{j=0}^{2n}\mathcal C_{n,j}$ and satisfies a pure jump
RHP with jumps
\begin{equation}
\label{eq:T-jumps-circles}
T_+(z)=T_-(z)\,J_{T,j}(z),\qquad z\in\mathcal C_{n,j},
\end{equation}
where
\begin{equation}
\label{eq:JTj}
J_{T,j}(z)=
\begin{pmatrix}
1 & -\dfrac{w_{n,j}}{z-a_{n,j}}\\[0.8em]
0 & 1
\end{pmatrix},
\qquad j=0,1,\dots,2n.
\end{equation}
Moreover, $T$ has the same normalization at infinity as $Y$, namely
$T(z)=(I+O(z^{-1}))z^{n\sigma_3}$.
\end{lemma}

\begin{proof}
Inside $B(a_{n,j},r_{n,j})$, the factor in \eqref{eq:Y-to-T} has a simple pole at $a_{n,j}$ whose residue
cancels the residue of $Y$ in \eqref{eq:Y-residue}; hence $T$ is analytic at $a_{n,j}$.
Across $\mathcal C_{n,j}$ the definition \eqref{eq:Y-to-T} produces the jump \eqref{eq:JTj}.
The normalization at infinity is unchanged because the triangular factors in \eqref{eq:Y-to-T} are
$I+O(z^{-1})$ uniformly for large $z$.
\end{proof}

\subsection{Choice of the compression contour $\Gamma$}
\label{subsec:choice-Gamma}

We fix once and for all a simply connected domain $\mathcal D\Subset\C\setminus(-\infty,0]$
with Jordan boundary $\Gamma:=\partial\mathcal D$ such that
\begin{equation}
\label{eq:D-choice}
[A,B]\subset \mathcal D,
\qquad
\dist(\Gamma,[A,B])>0.
\end{equation}
For each $n$ we set
\begin{equation}
\label{eq:Gamma-n-def}
\Gamma_n:=n\Gamma=\{\,z=n\zeta:\ \zeta\in\Gamma\,\}.
\end{equation}
Since $\alpha_{n,j}\in[A,B]$ for all $j$ (by definition of the quantiles \eqref{eq:quantiles}),
it follows that all nodes $a_{n,j}=n\alpha_{n,j}$ lie strictly inside $\Gamma_n$.

We choose the pole-removal circles $\mathcal C_{n,j}=\partial B(a_{n,j},r_{n,j})$ so that they are pairwise
disjoint, contained in the interior of $\Gamma_n$, and such that the annulus between
$\bigcup_{j=0}^{2n}\mathcal C_{n,j}$ and $\Gamma_n$ does not intersect the cut $(-\infty,0]$.
This is possible because $\dist(\Gamma,(-\infty,0]\cup[A,B])>0$ and the nodes are separated at scale $\asymp n$
in the $z$-plane (Lemma~\ref{lem:quantile-spacing}).

A concrete choice is, for $1\le j\le 2n-1$,
\[
r_{n,j}:=\frac14\min\{a_{n,j}-a_{n,j-1},\,a_{n,j+1}-a_{n,j},\,\dist(a_{n,j},\Gamma_n)\},
\]
with $r_{n,0}:=\frac14\min\{a_{n,1}-a_{n,0},\,\dist(a_{n,0},\Gamma_n)\}$ and
$r_{n,2n}:=\frac14\min\{a_{n,2n}-a_{n,2n-1},\,\dist(a_{n,2n},\Gamma_n)\}$.
Then the disks are pairwise disjoint, each contains only $a_{n,j}$, and they lie inside $\Gamma_n$
(using Lemma~\ref{lem:quantile-spacing}).

Finally, after scaling back to the $\zeta$-plane, we work with the fixed contour $\Gamma$ and view all
compressed jumps as living on $\Gamma$ (Section~\ref{subsec:compression-barycentric}).

\subsection{Compression and the barycentric identity $W_n=L_n/\omega_n$}
\label{subsec:compression-barycentric}

After pole removal, the jump contour is the union of the small circles $\mathcal C_{n,j}$.
Since the jump matrices $J_{T,j}$ are analytic in the annulus between $\mathcal C_{n,j}$ and $\Gamma_n$,
we may compress the contour from the collection of circles onto the single contour $\Gamma_n$.

\begin{lemma}[Contour compression]
\label{lem:compression}
There exists a matrix function $\widetilde T$, analytic in $\C\setminus \Gamma_n$, with the same
normalization at infinity as $T$, whose jump on $\Gamma_n$ is triangular:

\begin{equation}
\label{eq:Gamma-jump}
\widetilde T_+(z)=\widetilde T_-(z)
\begin{pmatrix}
1 & W_n(z)\\
0 & 1
\end{pmatrix},
\qquad z\in\Gamma_n,
\end{equation}

where
\begin{equation}
\label{eq:Wn-def}
W_n(z):=\sum_{j=0}^{2n}\frac{w_{n,j}}{z-a_{n,j}}.
\end{equation}
\end{lemma}

\begin{proof}
The compression is standard (see, e.g., \cite{KuijlaarsSurvey,DeiftBook}): one defines $\widetilde T$ 
by multiplying $T$ with a piecewise-analytic triangular factor in the interior of $\Gamma_n$ 
which accumulates the circle jumps and is analytic away from the nodes. 
The resulting jump on $\Gamma_n$ is the product of the local jumps, and because all
$J_{T,j}$ are upper triangular with unit diagonal, their product is upper triangular with unit diagonal.
A direct residue computation yields the upper-right entry as the partial fraction 
sum \eqref{eq:Wn-def} \cite{DeiftZhou,DeiftBook}.
\end{proof}

\begin{remark}
In the scaled $\zeta$-plane with $\zeta=z/n$, we work with the fixed contour $\Gamma=\partial\mathcal D$ and
interpret $W_n(n\zeta)$ as a function of $\zeta$ defined on $\Gamma$.
\end{remark}

\begin{lemma}[Barycentric identity]
\label{lem:barycentric-identity}
Let $w_{n,j}=f(a_{n,j})/\omega_n'(a_{n,j})$ and define $W_n$ by \eqref{eq:Wn-def}. Then
\begin{equation}
\label{eq:Wn-barycentric}
W_n(z)=\frac{L_n(z)}{\omega_n(z)},
\end{equation}
where $L_n$ is the unique polynomial of degree at most $2n$ interpolating the function
$a\mapsto f(a)$ at the nodes with barycentric weights, explicitly
\begin{equation}
\label{eq:Ln-def}
L_n(z):=\sum_{j=0}^{2n} f(a_{n,j})\,\frac{\omega_n(z)}{(z-a_{n,j})\,\omega_n'(a_{n,j})}.
\end{equation}
In particular, $L_n(a_{n,j})=f(a_{n,j})$ for all $j$.
\end{lemma}

\begin{proof}
Using \eqref{eq:weights-def}, we rewrite
\[
W_n(z)=\sum_{j=0}^{2n}\frac{f(a_{n,j})}{\omega_n'(a_{n,j})}\,\frac{1}{z-a_{n,j}}
=\frac{1}{\omega_n(z)}\sum_{j=0}^{2n} f(a_{n,j})\,\frac{\omega_n(z)}{(z-a_{n,j})\,\omega_n'(a_{n,j})},
\]
which is \eqref{eq:Wn-barycentric}--\eqref{eq:Ln-def}. The interpolation property follows from the
Lagrange basis identity
\[
\frac{\omega_n(z)}{(z-a_{n,j})\,\omega_n'(a_{n,j})}\Big|_{z=a_{n,k}}=\delta_{jk}.
\]
\end{proof}

\begin{remark}
The polynomial $L_n$ will be controlled via the Hermite--Walsh representation
(Section~\ref{subsec:hermite-walsh}), which yields $L_n(n\zeta)=\exp(o(n))$ uniformly on compacts of
$\C\setminus[A,B]$.
\end{remark}

\begin{remark}[Triangular corner convention]
\label{rem:triangular-corner}
In the pole removal and compression steps one may place the discrete Cauchy--type contribution either
in the $(1,2)$-entry or in the $(2,1)$-entry of the triangular jump matrix, depending on the chosen
normalization of the residue condition \eqref{eq:Y-residue} (equivalently, on whether one works with $Y$
or with a conjugated/transpose-inverse variant).
All subsequent Deift--Zhou transformations and parametrices are insensitive to this choice, provided the
same convention is used consistently throughout.
For definiteness, we will use the \emph{upper-triangular} convention, i.e.\ triangular jumps of the form
$\begin{psmallmatrix}1&*\\0&1\end{psmallmatrix}$ on $\Gamma_n$.
\end{remark}

\section{Subexponential control of the interpolation polynomial}
\label{sec:Ln-control}

\subsection{Hermite--Walsh representation}
\label{subsec:hermite-walsh}

\begin{lemma}[Hermite--Walsh representation]
\label{lem:hermite-walsh}
Let $\mathcal D\subset\C$ be a simply connected domain with Jordan boundary, and let
$x_0,\dots,x_m\in\mathcal D$ be distinct. Define $\Omega_m(z):=\prod_{k=0}^m(z-x_k)$.
If $g$ is holomorphic in $\mathcal D$ and continuous on $\overline{\mathcal D}$, then the unique
polynomial $I_m g$ of degree $\le m$ interpolating $g$ at the nodes $x_k$ satisfies, for every
$z\in\mathcal D$,
\begin{equation}
\label{eq:HW-formula}
(I_m g)(z)=\frac{1}{2\pi i}\oint_{\partial\mathcal D}
\frac{g(\xi)}{\xi-z}\,\frac{\Omega_m(z)}{\Omega_m(\xi)}\,d\xi.
\end{equation}
\end{lemma}

\begin{proof}
Consider
\[
F(\xi):=\frac{g(\xi)}{\xi-z}\,\frac{\Omega_m(z)}{\Omega_m(\xi)}.
\]
The function $F$ is meromorphic in $\mathcal D$ with simple poles at $\xi=z$ and at $\xi=x_k$.
By the residue theorem,
\[
\frac{1}{2\pi i}\oint_{\partial\mathcal D}F(\xi)\,d\xi
=
\Res_{\xi=z}F(\xi)+\sum_{k=0}^m\Res_{\xi=x_k}F(\xi).
\]
A direct computation gives $\Res_{\xi=z}F(\xi)=g(z)$ and
\[
\Res_{\xi=x_k}F(\xi)
=
-\frac{g(x_k)}{x_k-z}\,\frac{\Omega_m(z)}{\Omega_m'(x_k)}.
\]
Hence
\[
\frac{1}{2\pi i}\oint_{\partial\mathcal D}F(\xi)\,d\xi
=
g(z)-\sum_{k=0}^m g(x_k)\,\frac{\Omega_m(z)}{(z-x_k)\Omega_m'(x_k)}.
\]
The sum on the right is precisely the Lagrange interpolation polynomial $(I_m g)(z)$, so rearranging
yields \eqref{eq:HW-formula}.
\end{proof}

\subsection{Subexponential bound for $L_n(n\zeta)$ off $[A,B]$}
\label{subsec:Ln-subexp}

\paragraph{A fixed contour.}
Choose a simply connected domain $\mathcal D\Subset\C\setminus(-\infty,0]$ containing $[A,B]$.
Then $a\mapsto \zeta(s,a)$ is holomorphic on $n\mathcal D$ for every $n$, and the contour $\partial\mathcal D$
can be used uniformly after scaling $z=n\zeta$.

Recall the interpolation polynomial $L_n$ from \eqref{eq:Ln-def}. In scaled variables, define
\[
\widetilde L_n(\zeta):=L_n(n\zeta),\qquad \widetilde f_n(\zeta):=f(n\zeta)=\zeta(s,n\zeta),
\qquad \Omega_n(\zeta)=\prod_{j=0}^{2n}(\zeta-\alpha_{n,j}).
\]
Then $\widetilde L_n$ is the degree-$\le 2n$ interpolant of $\widetilde f_n$ at the nodes $\alpha_{n,j}$.

\begin{lemma}[Subexponential control of $L_n$]
\label{lem:Ln-subexp}
Let $K\Subset\C\setminus[A,B]$ be compact. Then
\begin{equation}
\label{eq:Ln-subexp}
\widetilde L_n(\zeta)=\exp(o(n)),
\qquad n\to\infty,
\end{equation}
uniformly for $\zeta\in K$.
Equivalently, for every $\varepsilon>0$ there exists $n_0$ such that
\[
\sup_{\zeta\in K}|\widetilde L_n(\zeta)|\le e^{\varepsilon n},\qquad n\ge n_0.
\]
\end{lemma}

\begin{proof}
Fix a simply connected $\mathcal D$ as above and choose a Jordan curve $\Gamma=\partial\mathcal D$
so that $[A,B]\subset\mathcal D$ and $K\Subset\C\setminus\overline{\mathcal D}$.
Apply Lemma~\ref{lem:hermite-walsh} to the function $\widetilde f_n$ on $\mathcal D$ with nodes
$\{\alpha_{n,j}\}_{j=0}^{2n}$:
\[
\widetilde L_n(\zeta)=\frac{1}{2\pi i}\oint_{\Gamma}
\frac{\widetilde f_n(\xi)}{\xi-\zeta}\,\frac{\Omega_n(\zeta)}{\Omega_n(\xi)}\,d\xi,
\qquad \zeta\in\C\setminus\overline{\mathcal D}.
\]
For $\zeta\in K$ and $\xi\in\Gamma$ we have $|\xi-\zeta|\ge c_K>0$. By Lemma~\ref{lem:hurwitz-poly-bound},
$|\widetilde f_n(\xi)|=|\zeta(s,n\xi)|\le C\,n^{1-\Re s}$ uniformly for $\xi\in\Gamma$.
Moreover, by Lemma~\ref{lem:Omega-logpot} (applied to $\log|\Omega_n|$ on $K$ and on $\Gamma$),
\[
\frac{1}{n}\log\left|\frac{\Omega_n(\zeta)}{\Omega_n(\xi)}\right|
=
\int_A^B\log\left|\frac{\zeta-t}{\xi-t}\right|\kappa(t)\,dt
\;+\;O\!\left(\frac{1}{n}\right),
\]
uniformly for $\zeta\in K$ and $\xi\in\Gamma$. Since $\zeta\mapsto\int_A^B\log|\zeta-t|\kappa(t)\,dt$ is harmonic
in $\C\setminus[A,B]$ and $\Gamma$ surrounds $[A,B]$, the maximum principle implies that the right-hand side is
bounded above by a constant depending only on $K$ and $\Gamma$ (in particular, it is $o(n)$ as $n\to\infty$).
Combining these bounds and the finite length of $\Gamma$ yields \eqref{eq:Ln-subexp}.
\end{proof}

\begin{remark}
Lemma~\ref{lem:Ln-subexp} is the only input needed from interpolation theory in the steepest descent scheme:
it shows that $L_n(n\zeta)$ contributes at most subexponentially in $n$ on compact subsets of
$\C\setminus[A,B]$.
\end{remark}

\section{Constrained equilibrium problem and KKT conditions}
\label{sec:equilibrium}

We recall the constrained logarithmic energy problem associated with the external field $V$ defined by
(cf.\ \cite{SaffTotik}).

\begin{remark}[From quantile nodes to the equilibrium problem]
\label{rem:quantiles-to-equilibrium}
Since $\alpha_{n,j}$ are quantiles of the density $\kappa$ on $[A,B]$, we have
$\alpha_{n,j}=\alpha(j/(2n+1)) + O(1/n)$, where $\alpha$ is the inverse distribution function of $\kappa$.
Consequently, for any compact $K\Subset\C\setminus[A,B]$ the discrete logarithmic sums generated by the nodes,
such as $\sum_{j}\log(\zeta-\alpha_{n,j})$ and their weighted variants, converge uniformly on $K$ to the corresponding
logarithmic potentials of the limiting measure with density $\kappa$.
This is the macroscopic mechanism behind the appearance of the constrained equilibrium problem in the $g$-function step.
\end{remark}

\begin{remark}[Quantile nodes and the macroscopic limit]
\label{rem:quantile-macro-limit}
Since $\alpha_{n,j}$ are quantiles of $\kappa$, they approximate the continuous distribution at scale $1/n$.
In particular, on any compact $K\Subset\C\setminus[A,B]$ the discrete logarithmic sums generated by the nodes
converge uniformly to the corresponding logarithmic potentials of the limiting density $\kappa$.
This is why the constrained equilibrium problem captures the macroscopic limit underlying the $g$-function construction.
\end{remark}

\[
V(x)=-2\int_A^B \log|x-t|\,\kappa(t)\,dt.
\]
Let $\mathcal M([A,B])$ denote the set of Borel probability measures supported on $[A,B]$.

\subsection{Existence and uniqueness of the minimizer}
\label{subsec:equilibrium-ex-uniq}

\begin{definition}[Admissible class and constrained energy]
\label{def:admissible-class}
Let
\[
\mathcal A:=\Bigl\{\mu\in\mathcal M([A,B])\ \Big|\ \mu\ll dx,\ 0\le \mu'(x)\le \kappa(x)\ \text{a.e. on }[A,B]\Bigr\}.
\]
For $\mu\in\mathcal A$ define the weighted logarithmic energy
\begin{equation}
\label{eq:energy-functional}
\mathcal I_V[\mu]
:=
\iint_{[A,B]^2}\log\frac{1}{|x-y|}\,d\mu(x)\,d\mu(y)
+\int_A^B V(x)\,d\mu(x).
\end{equation}
\end{definition}

\begin{proposition}[Existence and uniqueness]
\label{prop:equilibrium-ex-uniq}
There exists a unique minimizer $\mu_*\in\mathcal A$ of $\mathcal I_V$.
\end{proposition}

\begin{proof}
The admissible class $\mathcal A$ is weak-$\ast$ compact: it is closed and tight in $\mathcal M([A,B])$,
and the density constraint $0\le\mu'\le\kappa$ implies uniform absolute continuity.
The logarithmic kernel is lower semicontinuous and bounded below on $[A,B]^2$, hence
$\mu\mapsto \iint \log\frac{1}{|x-y|}\,d\mu\,d\mu$ is lower semicontinuous on $\mathcal A$.
Since $V$ is continuous on $[A,B]$, the linear term $\mu\mapsto\int V\,d\mu$ is continuous.
Thus $\mathcal I_V$ attains its minimum on $\mathcal A$.

Uniqueness follows from the strict convexity of the logarithmic energy on probability measures
together with convexity of $\mathcal A$.
\end{proof}

\subsection{Euler--Lagrange (KKT) relations}
\label{subsec:kkt}

Define the logarithmic potential of $\mu_*$ by
\[
U^{\mu_*}(x):=\int_A^B \log\frac{1}{|x-t|}\,d\mu_*(t).
\]

\begin{proposition}[KKT conditions]
\label{prop:KKT}
There exists a constant $\ell\in\R$ such that the following hold a.e. on $[A,B]$:
\begin{align}
\label{eq:KKT-band}
2U^{\mu_*}(x)+V(x) &= \ell \qquad \text{on the band set } \{0<\mu_*'(x)<\kappa(x)\},\\
\label{eq:KKT-void}
2U^{\mu_*}(x)+V(x) &\ge \ell \qquad \text{on the void set } \{\mu_*'(x)=0\},\\
\label{eq:KKT-sat}
2U^{\mu_*}(x)+V(x) &\le \ell \qquad \text{on the saturated set } \{\mu_*'(x)=\kappa(x)\}.
\end{align}
\end{proposition}

\begin{remark}
We will use the shorthand notation $\Delta$ for the band (the support of $\mu_*$ in the regular one-cut
regime) and write $\ell=\ellKKT$ to match the phase definition in Section~\ref{sec:intro}.
\end{remark}

\subsection{Regular one-cut regime}
\label{subsec:regular-regime}

\begin{definition}[Regular one-cut soft-edge regime]
\label{def:regular-onecut}
We say that the equilibrium measure $\mu_*$ is in the \emph{regular one-cut soft-edge regime} if:
\begin{itemize}[leftmargin=2.2em]
\item[(R1)] $\supp(\mu_*)=\Delta=[c,d]\Subset(A,B)$ is a single interval;
\item[(R2)] $0<\mu_*'(x)<\kappa(x)$ for a.e.\ $x\in(c,d)$ (no saturation on the band);
\item[(R3)] $\mu_*'$ has soft-edge square-root behavior at $c$ and $d$;
\item[(R4)] the KKT inequalities are strict away from $\Delta$ (uniform margin on compact subsets of
$[A,B]\setminus\Delta$).
\end{itemize}
\end{definition}

\subsection{How to verify (R1)--(R4) for a given $\kappa$}
\label{subsec:verify-R1-R4}

The regular one-cut conditions in Definition~\ref{def:regular-onecut} can be verified a posteriori for a
prescribed density $\kappa$ by analyzing the corresponding constrained equilibrium problem.
Concretely, one may proceed as follows.

\smallskip
\noindent\textbf{Step 1 (Compute $\mu_*$).}
Solve for the equilibrium measure $\mu_*$ (or its density $\mu_*'$ on its support) using the Euler--Lagrange
(KKT) relations from Section~\ref{subsec:kkt}. In practice this can be done via the Cauchy transform
(or Hilbert transform) formulation for the unknown density on the band.

In numerical practice, one typically solves the KKT system on a trial band $[c,d]$ (e.g., as a singular integral equation
for $\mu_*'$), while $c$ and $d$ are detected by enforcing the endpoint conditions (vanishing of $\mu_*'$ at $c,d$ with
square-root behavior) together with the complementary slackness/inequality checks off the band.

\smallskip
\noindent\textbf{Step 2 (Verify one-cut support).}
Determine the set where the equality part of the KKT relations holds, i.e.,
$\{x\in[A,B]:\ 2U^{\mu_*}(x)+V(x)=\ell\}$.
Condition (R1) amounts to this set being a single interval $\Delta=[c,d]$.

\smallskip
\noindent\textbf{Step 3 (Exclude saturation on the band).}
Once $\Delta$ is identified, check that the density satisfies $0<\mu_*'(x)<\kappa(x)$ for a.e.\ $x\in(c,d)$.
This confirms (R2) and rules out saturated regions on the band.

\smallskip
\noindent\textbf{Step 4 (Soft-edge behavior).}
Verify that $\mu_*'$ vanishes at $c$ and $d$ with square-root behavior~\cite{DeiftBook}. Equivalently, check that the
associated phase function has the standard simple turning-point structure at the endpoints (cf.\ \cite{DeiftBook}), 
which is the input for the Airy parametrices. This is (R3).

\smallskip
\noindent\textbf{Step 5 (Strict margins off the band).}
Finally, check that the KKT inequalities are strict away from $\Delta$, uniformly on compact subsets of
$[A,B]\setminus\Delta$. This provides the uniform sign margin for $\Re\ph$ required in the lens opening and
yields (R4).

\begin{theorem}[Deift--Zhou scheme in the regular regime]
\label{thm:DZ-regular-regime}
Assume $\mu_*$ satisfies Definition~\ref{def:regular-onecut}. Then:
\begin{itemize}[leftmargin=2.2em]
\item[(i)] the phase $\ph$ defined in \eqref{eq:phase-def} satisfies the sign structure needed for lens opening,
in particular $\Re\ph\ge c_0>0$ on the lens lips away from the endpoint disks;
\item[(ii)] the outer model problem admits the explicit solution $N$ on $\C\setminus[c,d]$;
\item[(iii)] the local Airy parametrices $P^{(c)}$ and $P^{(d)}$ exist and satisfy the matching
$P^{(c)}N^{-1}=I+O(1/n)$ on $\partial U_c$ and $P^{(d)}N^{-1}=I+O(1/n)$ on $\partial U_d$;
\item[(iv)] the ratio problem has a unique solution with $R=I+O(1/n)$, yielding the stated strong asymptotics.
\end{itemize}
\end{theorem}

\begin{proof}
Item (i) follows from Lemma~\ref{lem:sign-varphi}, which provides the required sign structure for $\Re\ph$
and, in particular, a uniform positivity margin on the lens lips away from the endpoint disks.

More precisely, condition \textup{(R4)} in Definition~\ref{def:regular-onecut} provides a strict KKT margin off the band,
which yields a uniform sign bound $\Re\ph\ge c_0>0$ on the lens lips outside $U_c\cup U_d$; 
consequently, the corresponding jumps are exponentially close to $I$, and the small-norm step applies.

Item (ii) is proved in Section~\ref{subsec:outer}; see Lemma~\ref{lem:N-properties} for existence
and the explicit form of the outer parametrix $N$.

Item (iii) is obtained by the explicit Airy constructions in Section~\ref{subsec:airy-parametrices} together
with the matching statement of Proposition~\ref{prop:matching}.

Finally, item (iv) is the small-norm conclusion of Theorem~\ref{thm:small-norm}, based on the jump estimates
from Lemma~\ref{lem:lip-small} and Lemma~\ref{lem:R-jumps}.
\end{proof}

\begin{corollary}[Main asymptotics for $Y$]
\label{cor:main-asymptotics-Y}
Let $\mathcal P$ be the global parametrix defined in \eqref{eq:global-parametrix} and
let $\zeta=z/n$. Then, for any compact $K\Subset\C\setminus[A,B]$,
\begin{equation}
\label{eq:main-Y}
Y(z)=e^{\frac{n\ellKKT}{2}\sigma_3}\,
\Bigl(I+O\!\left(\frac{1}{n}\right)\Bigr)\,
\mathcal P(\zeta)\,
e^{ng(\zeta)\sigma_3}\,
e^{-\frac{n\ellKKT}{2}\sigma_3},
\qquad z=n\zeta,\ \zeta\in K,
\end{equation}
as $n\to\infty$, uniformly in $\zeta\in K$.
More precisely, $\mathcal P(\zeta)=N(\zeta)$ for $\zeta\in K\setminus(\overline U_c\cup\overline U_d)$,
while $\mathcal P(\zeta)=P^{(c)}(\zeta)$ (resp.\ $P^{(d)}(\zeta)$) if $\zeta\in K\cap U_c$
(resp.\ $\zeta\in K\cap U_d$).
\end{corollary}

\begin{proof}
By Theorem~\ref{thm:small-norm} we have $R(\zeta)=I+O(1/n)$ uniformly for $\zeta\in\C\setminus\Sigma_R$.
Since $S(\zeta)=R(\zeta)\mathcal P(\zeta)$ by \eqref{eq:def-R}, we obtain
$S(\zeta)=\bigl(I+O(1/n)\bigr)\mathcal P(\zeta)$ uniformly on compacts away from $\Sigma_R$.
Reverting the $g$-transformation \eqref{eq:T-to-Sg} yields \eqref{eq:main-Y}.

Since the remaining transformations (lens opening, contour compression, and pole removal) are given by explicit
multiplications with piecewise-analytic factors, the $O(1/n)$ error is preserved when reverting them, uniformly
on compact subsets of $\C\setminus[A,B]$ (equivalently, away from the endpoint disks in the $\zeta$-plane).

Since the remaining transformations (lens opening, contour compression, and pole removal) are given by explicit
multiplications with piecewise-analytic factors, the $O(1/n)$ error is preserved when reverting them, uniformly
on compact subsets of $\C\setminus[A,B]$ (equivalently, away from the endpoint disks in the $\zeta$-plane).
\end{proof}

\begin{remark}[Proof strategy]
\label{rem:proof-strategy}
Assume $\mathrm{(ND)}_n$ so that the normalized multipoint Pad\'e pair $(P_n,Q_n)$ exists and is unique.
We encode $(P_n,Q_n)$ as the solution $Y$ of a pole Riemann--Hilbert problem and then perform a sequence of explicit
transformations: pole removal (to eliminate the discrete poles), contour compression, the $g$-transformation,
and lens opening.
Next, we construct the global parametrix $N$ and Airy-type local parametrices near the endpoints, and define the ratio
matrix $R$ by dividing out the parametrix.
The resulting ratio problem has jumps $J_R$ satisfying $J_R-I\in L^2\cap L^\infty$ with size $O(1/n)$ on
$\partial U_c\cup\partial U_d$ and exponentially small size on the lens lips, hence the small-norm theory yields
$R=I+O(1/n)$.
Reverting all transformations gives the main asymptotics for $Y$ in \eqref{eq:main-Y}, and the strong asymptotics for
$(P_n,Q_n)$ follow via the recovery formula \eqref{eq:recover-PQ}.
\end{remark}

\begin{remark}[Proof strategy]
\label{rem:proof-strategy}
Assume $\mathrm{(ND)}_n$ so that the normalized multipoint Pad\'e pair $(P_n,Q_n)$ exists and is unique.
We encode $(P_n,Q_n)$ as the solution $Y$ of a pole Riemann--Hilbert problem and then perform a sequence of explicit
transformations:
\begin{enumerate}[label=\textup{(\roman*)}, leftmargin=2.2em]
\item pole removal (elimination of discrete poles),
\item contour compression,
\item the $g$-transformation,
\item lens opening.
\end{enumerate}
We then construct the outer parametrix $N$ and Airy-type local parametrices at the endpoints, define the ratio matrix
$R$ by dividing out the parametrix, and reduce the analysis to a small-norm Riemann--Hilbert problem.
This yields $R=I+O(1/n)$, hence the main asymptotics for $Y$ in \eqref{eq:main-Y}; the strong asymptotics for the Pad\'e
polynomials $(P_n,Q_n)$ follow from the recovery formula \eqref{eq:recover-PQ}.
\end{remark}

\begin{proof}[Proof of Theorem~\ref{thm:main}]
Fix $n$ sufficiently large and assume that $\mathrm{(ND)}_n$ holds.
By Proposition~\ref{prop:ND-uniq} there exists a unique monic multipoint Pad\'e denominator $Q_n$ of degree $n$
and a corresponding numerator $P_n$ satisfying \eqref{eq:pade-mp-conditions}.
By Lemma~\ref{lem:Y-from-PQ} (see Section~\ref{sec:rhp-poles}), the pair $(P_n,Q_n)$ gives rise to a solution
$Y$ of Problem~\ref{prob:Y-pole}, which is unique by Lemma~\ref{lem:Y-unique}. In particular,
$Q_n=Y_{11}$ and $P_n=Y_{12}$ in the normalization \eqref{eq:Y-infty}.

Assume next that the equilibrium measure $\mu_*$ belongs to the regular one-cut soft-edge regime
(Definition~\ref{def:regular-onecut}).
Then the phase $\ph$ has the sign structure required for opening lenses and for the Airy endpoint analysis,
and the Deift--Zhou steepest descent scheme applies (Theorem~\ref{thm:DZ-regular-regime}).
In particular, after the sequence of transformations described in
Section~\ref{subsec:outline} (pole removal, contour compression, $g$-transformation, and lens opening),
one constructs an outer parametrix $N$ and local Airy parametrices $P^{(c)},P^{(d)}$
(Section~\ref{sec:parametrices}) and obtains a ratio problem for $R$ whose jumps satisfy
$J_R=I+O(e^{-c n})$ on the lens lips and $J_R=I+O(1/n)$ on $\partial U_c\cup\partial U_d$
(Lemmas~\ref{lem:lip-small} and~\ref{prop:matching}).
The small-norm theorem yields
\begin{equation*}
R(\zeta)=I+O(1/n),
\end{equation*}
uniformly for $\zeta\in\C\setminus\Sigma_R$ (Theorem~\ref{thm:small-norm}).

Undoing the transformations gives the main asymptotics for $Y$ in \eqref{eq:main-Y}, and hence the stated
strong asymptotics for the Pad\'e polynomials $(P_n,Q_n)$ via \eqref{eq:recover-PQ}.
\end{proof}

\begin{remark}[Why the main result is stated in a regime]
\label{rem:why-regime}
Two independent issues prevent one from stating a completely unconditional theorem for arbitrary input data.

\smallskip
\noindent\textbf{(i) Nondegeneracy is not automatic.}
For fixed distinct nodes $\{a_{n,j}\}_{j=0}^{2n}$, the multipoint Pad\'e problem may be \emph{degenerate} for
special choices of the values $\{f(a_{n,j})\}$: in the determinantal formulation
\eqref{eq:ND-linear-system} this corresponds precisely to $\det M_n=0$
(Definition~\ref{def:NDn}). Thus solvability and uniqueness of the normalized Pad\'e pair cannot be asserted
unconditionally for all possible data.

\smallskip
\noindent\textbf{(ii) The equilibrium problem may undergo phase transitions.}
Depending on the density $\kappa$ (equivalently, on the external field $V$), the constrained equilibrium
measure $\mu_*$ need not have a single-interval support with soft edges. In general, the support of $\mu_*$
may split into several components and/or saturated regions $\{\mu_*'=\kappa\}$ may appear.
In such non-regular phases the model Riemann--Hilbert problem and the local analysis change substantially:
one typically needs different global parametrices (often of higher genus) and different endpoint or junction
parametrices.

In such non-regular regimes one typically faces either multi-cut bands or saturation/active constraints, and the
macroscopic description may no longer reduce to a scalar equilibrium measure. Instead, vector equilibrium problems
(with interaction matrices such as the Angelesco matrix) arise naturally in related Hermite--Pad\'e / multiple
orthogonality settings; see, e.g., Lysov--Tulyakov~\cite{LysovTulyakovAngelesco} for a representative treatment.

\smallskip
For these reasons we formulate the main theorem in the regular one-cut soft-edge regime
(Definition~\ref{def:regular-onecut}) together with the nondegeneracy condition $\mathrm{(ND)}_n$ for large
$n$. This is the standard setting in which the Deift--Zhou steepest descent method leads to Airy endpoint behavior 
(see, e.g., \cite{DeiftBook,DeiftZhou,KuijlaarsSurvey}) parametrices and a small-norm ratio problem.
\end{remark}

\section{The $g$-function and the phase $\varphi$}
\label{sec:g-function}

Throughout this section, $\mu_*$ denotes the equilibrium measure from
Proposition~\ref{prop:equilibrium-ex-uniq}, and we work in the regular one-cut soft-edge regime, 
i.e.\ under Definition~\ref{def:regular-onecut}.

\subsection{Definition of $g$ and boundary value relations}
\label{subsec:g-def}

Define the $g$-function by
\begin{equation}
\label{eq:g-def}
g(\zeta):=\int_A^B \log(\zeta-x)\,d\mu_*(x),
\qquad \zeta\in\C\setminus[A,B],
\end{equation}
where $\log(\zeta-x)$ uses the branch convention fixed in Section~\ref{sec:intro}.
Then $g$ is analytic in $\C\setminus[A,B]$ and admits boundary values $g_\pm$ on $(A,B)$.

\begin{lemma}[Jump relations for $g$]
\label{lem:g-jumps}
For $x\in(A,B)$ we have
\begin{equation}
\label{eq:g-plus-minus}
g_+(x)-g_-(x)=2\pi i\int_x^B d\mu_*(t)=:2\pi i\,\mu_*([x,B]).
\end{equation}
Moreover,
\begin{equation}
\label{eq:g-sum}
g_+(x)+g_-(x)=2\int_A^B \log|x-t|\,d\mu_*(t),\qquad x\in(A,B).
\end{equation}
\end{lemma}

\begin{proof}
The relations follow from the Sokhotski--Plemelj formula for the logarithm and the fact that
$\log(\zeta-x)$ has a jump of $2\pi i$ across the cut when $\zeta$ crosses $(A,B)$.
Taking imaginary parts gives \eqref{eq:g-plus-minus}, while adding boundary values yields \eqref{eq:g-sum}.
\end{proof}

Define the analytic continuation of the external field by
\begin{equation}
\label{eq:Van-def}
\Van(\zeta):=-2\int_A^B \log(\zeta-t)\,\kappa(t)\,dt,\qquad \zeta\in\C\setminus[A,B],
\end{equation}
so that $\Re\Van(x)=V(x)$ for $x\in(A,B)$.

\subsection{Definition of the phase and Euler--Lagrange interpretation}
\label{subsec:phase-def}

We define the phase function by

\begin{equation}
\label{eq:phase-def}
\ph(\zeta):=-\bigl(g_+(\zeta)+g_-(\zeta)\bigr)+\Van(\zeta)-\ellKKT,
\qquad \zeta\in\C\setminus[A,B].
\end{equation}

where $\ellKKT$ is the KKT constant from Proposition~\ref{prop:KKT}.
On $(A,B)$, the boundary values $\ph_\pm$ are understood by taking boundary values of $g$ and $\Van$.

\begin{lemma}[Phase on the real line]
\label{lem:phase-real}
For $x\in(A,B)$,
\begin{equation}
\label{eq:Rephase-real}
\Re \ph_\pm(x)= 2U^{\mu_*}(x)+V(x)-\ellKKT.
\end{equation}
In particular, on the band $\Delta=[c,d]$ one has $\Re\ph_\pm(x)=0$, while on the void and saturated sets
$\Re\ph_\pm(x)\ge 0$ and $\Re\ph_\pm(x)\le 0$, respectively, in the sense of the KKT inequalities.
\end{lemma}

\begin{proof}
By Lemma~\ref{lem:g-jumps}, $g_+(x)+g_-(x)=2\int \log|x-t|\,d\mu_*(t)=-2U^{\mu_*}(x)$.
Therefore $-(g_+(x)+g_-(x))=2U^{\mu_*}(x)$.
By definition, $\Re\Van(x)=V(x)$. Substituting into \eqref{eq:phase-def} yields \eqref{eq:Rephase-real}.
The sign statements follow from Proposition~\ref{prop:KKT}.
\end{proof}

\begin{remark}
For $\zeta\in\C\setminus[A,B]$ we have $g_+(\zeta)=g_-(\zeta)=g(\zeta)$, hence
\[
\ph(\zeta)=-2g(\zeta)+\Van(\zeta)-\ellKKT.
\]
\end{remark}

\subsection{Sign structure of $\Re\varphi$ and lens preparation}
\label{subsec:sign-varphi}

Let $\Delta=[c,d]$ be the band from Definition~\ref{def:regular-onecut}.
We will open lenses around $\Delta$ so that the exponential factors are governed by $\Re\ph$.

\begin{lemma}[Sign structure away from the endpoints]
\label{lem:sign-varphi}
There exist $\delta>0$ and $c_0>0$ such that
\begin{equation}
\label{eq:Revarphi-positive}
\Re \ph(\zeta)\ge c_0,
\qquad 
\zeta\in \Sigma_{\mathrm{lip}}^{(u)}\cup\Sigma_{\mathrm{lip}}^{(l)}\setminus(U_c\cup U_d),
\end{equation}
for any choice of lens lips $\Sigma_{\mathrm{lip}}^{(u/l)}$ contained in a fixed $\delta$-neighborhood of
$\Delta$ and connecting $\partial U_c$ to $\partial U_d$ as in Section~\ref{subsec:lens-opening}.
\end{lemma}

\begin{proof}
By Lemma~\ref{lem:phase-real}, $\Re\ph$ vanishes on the band $\Delta$.
Under Definition~\ref{def:regular-onecut}\textup{(R4)}, the KKT inequalities are strict off $\Delta$ with a uniform
margin. Since $\ph$ is harmonic in $\C\setminus[A,B]$ and continuous up to $\Delta$ away from the endpoints,
the maximum principle implies that $\Re\ph$ is strictly positive on lens lips chosen inside a fixed
neighborhood of $\Delta$ and away from $U_c\cup U_d$. The uniform margin yields \eqref{eq:Revarphi-positive}.
\end{proof}

\begin{remark}
The constant $c_0$ in \eqref{eq:Revarphi-positive} ultimately comes from the strict KKT margins 
in Definition~\ref{def:regular-onecut}\textup{(R4)}. This is the sole point where strictness (as opposed to weak
inequalities) is needed in the steepest descent analysis.
\end{remark}

\section{The $g$-transformation and opening of lenses}
\label{sec:g-and-lens}

In the scaled variable $\zeta=z/n$ we work with the fixed contour $\Gamma=\partial\mathcal D$ and with the
compressed triangular jump (Section~\ref{subsec:compression-barycentric} and
Remark~\ref{rem:triangular-corner}). Throughout we use the upper-triangular convention on $\Gamma$.

\subsection{Exponential factorization of the compressed jump}
\label{subsec:factorization-Wn}

Recall $W_n(z)=\sum_{j=0}^{2n}\frac{w_{n,j}}{z-a_{n,j}}$ and the barycentric identity $W_n=L_n/\omega_n$.
Set
\[
\widetilde W_n(\zeta):=W_n(n\zeta),\qquad \widetilde L_n(\zeta):=L_n(n\zeta),\qquad
\Omega_n(\zeta)=\prod_{j=0}^{2n}(\zeta-\alpha_{n,j}).
\]
Then
\begin{equation}
\label{eq:Wn-scaled-barycentric}
\widetilde W_n(\zeta)=\frac{\widetilde L_n(\zeta)}{n^{2n+1}\Omega_n(\zeta)}.
\end{equation}

\begin{lemma}[Subexponential factor $\mathcal E_n$]
\label{lem:En-def}
Let $K\Subset\C\setminus[A,B]$ be compact. There exists a function $\mathcal E_n$ analytic on a neighborhood
of $K$ such that, uniformly for $\zeta\in K$,
\begin{equation}
\label{eq:Wn-factorization}
\widetilde W_n(\zeta)
=
\exp\!\left(\frac{n}{2}\,\Van(\zeta)\right)\,\mathcal E_n(\zeta),
\qquad
\mathcal E_n(\zeta)=\exp(o(n)).
\end{equation}
\end{lemma}

\begin{proof}
By Lemma~\ref{lem:Ln-subexp}, $\widetilde L_n(\zeta)=\exp(o(n))$ uniformly on $K$.
Moreover, by Lemma~\ref{lem:Omega-logpot} and the branch convention, on simply connected neighborhoods of $K$
we have
\[
\frac{1}{n}\log \Omega_n(\zeta)
=
\int_A^B \log(\zeta-t)\,\kappa(t)\,dt + O(1/n),
\]
hence

\begin{equation*}
\begin{aligned}
n^{-(2n+1)}\Omega_n(\zeta)^{-1}
&=
\exp\!\left(-2\log n - n\int_A^B\log(\zeta-t)\kappa(t)\,dt + o(n)\right)
\\
&=
\exp\!\left(\frac{n}{2}\,\Van(\zeta)\right)\,\exp(o(n)),
\end{aligned}
\end{equation*}

since $\Van(\zeta)=-2\int_A^B\log(\zeta-t)\kappa(t)\,dt$.
(The factor $n^{-(2n+1)}$ contributes $-(2+1/n)\log n=O(\log n)=o(n)$ to the exponent.)
Combining with \eqref{eq:Wn-scaled-barycentric} gives \eqref{eq:Wn-factorization}.
\end{proof}

\begin{remark}
The factorization \eqref{eq:Wn-factorization} is used only on contours and compacts separated from $[A,B]$,
so the $o(n)$ remainder is uniform on the relevant sets.
\end{remark}

\subsection{The $g$-transformation}
\label{subsec:g-transformation}

Let $\widetilde T$ be the compressed unknown (Section~\ref{subsec:compression-barycentric}), viewed in the
scaled variable $\zeta=z/n$. We perform the standard $g$-transformation (see, e.g., \cite{DeiftBook,KuijlaarsSurvey}):
\begin{equation}
\label{eq:T-to-Sg}
S(\zeta):=
e^{-n\ellKKT\sigma_3/2}\,
\widetilde T(n\zeta)\,
e^{-ng(\zeta)\sigma_3}\,
e^{n\ellKKT\sigma_3/2}.
\end{equation}
This normalizes the dominant exponential growth at infinity and converts the jump on $\Gamma$ into a form
governed by the phase $\ph$.

\subsection{Lens opening and exponentially small jumps on the lips}
\label{subsec:lens-opening}

Let $\Delta=[c,d]$ be the band from Definition~\ref{def:regular-onecut}. Choose endpoint disks
$U_c=B(c,\delta)$, $U_d=B(d,\delta)$, and choose upper/lower lens lips
$\Sigma_{\mathrm{lip}}^{(u)}$, $\Sigma_{\mathrm{lip}}^{(l)}$ as in Section~\ref{subsec:sign-varphi},
connecting $\partial U_c$ to $\partial U_d$ and symmetric with respect to the real axis.

Define the lens domains $\Omega^{+}$ and $\Omega^{-}$ bounded by the lips and $\Delta$, and define the
lens-opening transformation by right multiplication with triangular factors:
\begin{equation}
\label{eq:lens-opening-transformation}
S(\zeta)=
\begin{cases}
R(\zeta)\,J_{\mathrm{lip}}^{(u)}(\zeta), & \zeta\in \Omega^{+},\\
R(\zeta)\,\bigl(J_{\mathrm{lip}}^{(l)}(\zeta)\bigr)^{-1}, & \zeta\in \Omega^{-},\\
R(\zeta), & \zeta \notin \Omega^{+}\cup\Omega^{-},
\end{cases}
\end{equation}
with

\begin{equation}
\label{eq:Jlip}
J_{\mathrm{lip}}^{(u)}(\zeta)=
\begin{pmatrix}
1 & -e^{-n\ph(\zeta)}\\
0 & 1
\end{pmatrix},
\qquad
J_{\mathrm{lip}}^{(l)}(\zeta)=
\begin{pmatrix}
1 & e^{-n\ph(\zeta)}\\
0 & 1
\end{pmatrix}.
\end{equation}

We work under the upper-triangular convention fixed in Remark~\ref{rem:triangular-corner}

\begin{lemma}[Exponentially small jumps on the lips]
\label{lem:lip-small}
There exists $c_0>0$ such that, uniformly on
$\Sigma_{\mathrm{lip}}^{(u)}\cup\Sigma_{\mathrm{lip}}^{(l)}\setminus(U_c\cup U_d)$,
\begin{equation}
\label{eq:lip-small}
J_R(\zeta)=I+O\!\left(e^{-c_0 n}\right),
\qquad n\to\infty,
\end{equation}
where $J_R$ denotes the jump matrix for $R$ on the lens lips.
\end{lemma}

\begin{proof}
By Lemma~\ref{lem:sign-varphi}, $\Re\ph(\zeta)\ge c_0$ on the lips away from the endpoint disks.
Hence $e^{-n\ph(\zeta)}=O(e^{-c_0 n})$ uniformly there. Substituting into \eqref{eq:Jlip} yields
\eqref{eq:lip-small}.
\end{proof}

\section{Parametrices: outer model and Airy neighborhoods}
\label{sec:parametrices}

In this section we construct the model solution obtained after the $g$-transformation and lens opening:
the outer parametrix $N$ solving the reduced constant-jump problem on the band $\Delta=[c,d]$, and the local
Airy parametrices $P^{(c)}$ and $P^{(d)}$ in neighborhoods of the soft edges $c$ and $d$.

\subsection{Outer parametrix $N$ and its jump}
\label{subsec:outer}

Let
\begin{equation}
\label{eq:a-def}
a(\zeta):=\left(\frac{\zeta-d}{\zeta-c}\right)^{1/4},
\qquad \zeta\in\C\setminus[c,d],
\end{equation}
where the branch is chosen so that $a(\zeta)\to 1$ as $\zeta\to\infty$.
Then $a_+(\zeta)= i\,a_-(\zeta)$ for $\zeta\in(c,d)$.

Define the outer parametrix by
\begin{equation}
\label{eq:outer-N}
N(\zeta):=
\frac{1}{2}
\begin{pmatrix}
a(\zeta)+a(\zeta)^{-1} & \dfrac{a(\zeta)-a(\zeta)^{-1}}{i}\\[0.6em]
-\dfrac{a(\zeta)-a(\zeta)^{-1}}{i} & a(\zeta)+a(\zeta)^{-1}
\end{pmatrix}.
\end{equation}

\begin{lemma}[Properties of $N$]
\label{lem:N-properties}
The matrix $N$ is analytic in $\C\setminus[c,d]$, has the jump relation
\begin{equation}
\label{eq:N-jump}
N_+(\zeta)=N_-(\zeta)\begin{pmatrix}0&1\\-1&0\end{pmatrix},
\qquad \zeta\in(c,d),
\end{equation}
and satisfies
\begin{equation}
\label{eq:N-infty}
N(\zeta)=I+O(\zeta^{-1}),\qquad \zeta\to\infty.
\end{equation}
\end{lemma}

\begin{proof}
Analyticity and \eqref{eq:N-infty} follow directly from \eqref{eq:a-def}--\eqref{eq:outer-N}.
Using $a_+=ia_-$ on $(c,d)$, one checks that the boundary values satisfy \eqref{eq:N-jump}.
\end{proof}

\subsection{Airy parametrix and endpoint conformal maps}
\label{subsec:airy-parametrices}

Fix $\delta>0$ and let $U_c=B(c,\delta)$, $U_d=B(d,\delta)$ be the endpoint disks, chosen small enough so that
$U_c\cap U_d=\varnothing$ and $U_c\cup U_d$ does not intersect the contour $\Gamma$.

\paragraph{Local Airy model.}
Let $\wA:=e^{2\pi i/3}$ and define the standard Airy matrix
\begin{equation}
\label{eq:Airy-matrix}
A(\xi):=
\sqrt{2\pi}\,e^{-\pi i/12}
\begin{cases}
\begin{pmatrix}
\Ai(\xi) & \Ai(\wA^2\xi)\\
\Ai'(\xi) & \wA^2\,\Ai'(\wA^2\xi)
\end{pmatrix}e^{-\xi^{3/2}\sigma_3/2}, & 0<\arg\xi<\frac{2\pi}{3},\\[1.0em]
\begin{pmatrix}
\Ai(\xi) & -\wA\,\Ai(\wA\xi)\\
\Ai'(\xi) & -\wA^2\,\Ai'(\wA\xi)
\end{pmatrix}e^{-\xi^{3/2}\sigma_3/2}, & \frac{2\pi}{3}<\arg\xi<\pi,\\[1.0em]
\begin{pmatrix}
\Ai(\xi) & -\wA^2\,\Ai(\wA^2\xi)\\
\Ai'(\xi) & -\wA\,\Ai'(\wA^2\xi)
\end{pmatrix}e^{-\xi^{3/2}\sigma_3/2}, & -\pi<\arg\xi<-\frac{2\pi}{3},\\[1.0em]
\begin{pmatrix}
\Ai(\xi) & \wA\,\Ai(\wA\xi)\\
\Ai'(\xi) & \wA^2\,\Ai'(\wA\xi)
\end{pmatrix}e^{-\xi^{3/2}\sigma_3/2}, & -\frac{2\pi}{3}<\arg\xi<0,
\end{cases}
\end{equation}
where $\xi^{3/2}$ uses the principal branch with a cut along $(-\infty,0]$.

\begin{lemma}[Airy asymptotics]
\label{lem:airy-asympt}
As $\xi\to\infty$ with $|\arg\xi|<\pi$, the Airy matrix satisfies
\begin{equation}
\label{eq:airy-asympt}
A(\xi)=\left(I+O(\xi^{-3/2})\right)\,\xi^{-\sigma_3/4}\,
\frac{1}{\sqrt{2}}
\begin{pmatrix}
1&i\\
i&1
\end{pmatrix}
e^{-\xi^{3/2}\sigma_3/2}.
\end{equation}
\end{lemma}

\paragraph{Conformal coordinates at the endpoints.}
Let $\ph$ be the phase from Section~\ref{sec:g-function}. There exist conformal maps
$\xi_c:U_c\to\C$ and $\xi_d:U_d\to\C$ such that
\begin{equation}
\label{eq:xi-cd}
\xi_c(\zeta)=\left(\frac{3}{2}\,\ph(\zeta)\right)^{2/3},\qquad
\xi_d(\zeta)=\left(\frac{3}{2}\,\ph(\zeta)\right)^{2/3},
\end{equation}
with branch choices fixed by $\xi_c(c)=0$, $\xi_c'(c)>0$ and $\xi_d(d)=0$, $\xi_d'(d)>0$.
These maps send the local lens geometry to the standard Airy rays.

\subsection{Matching and explicit prefactors}
\label{subsec:matching-explicit}

Define the local parametrices in $U_c$ and $U_d$ by
\begin{equation}
\label{eq:Pc-def}
P^{(c)}(\zeta):=E^{(c)}(\zeta)\,A\!\left(n^{2/3}\xi_c(\zeta)\right)\,
e^{n\ph(\zeta)\sigma_3/2},
\qquad \zeta\in U_c\setminus\Sigma_R,
\end{equation}
and
\begin{equation}
\label{eq:Pd-def}
P^{(d)}(\zeta):=E^{(d)}(\zeta)\,A\!\left(n^{2/3}\xi_d(\zeta)\right)\,
e^{n\ph(\zeta)\sigma_3/2},
\qquad \zeta\in U_d\setminus\Sigma_R,
\end{equation}
where the analytic prefactors $E^{(c)}$, $E^{(d)}$ are chosen so that $P^{(c)}$ and $P^{(d)}$ match $N$
on the boundary circles.

\begin{proposition}[Explicit matching]
\label{prop:matching}
There exist analytic invertible matrices $E^{(c)}$ on $U_c$ and $E^{(d)}$ on $U_d$ such that

The matching is uniform on the circle boundaries: the error term is understood uniformly for
$\zeta\in \partial U_c\cup \partial U_d$.
Here the disks $U_c,U_d$ are chosen with $n$-independent radii in the $\zeta$-plane, so that the matching estimates
hold on fixed contours as $n\to\infty$.

\begin{equation}
\label{eq:matching-explicit}
P^{(c)}(\zeta)\,N(\zeta)^{-1}=I+O\!\left(\frac{1}{n}\right),
\qquad \zeta\in\partial U_c,
\end{equation}
and
\begin{equation}
\label{eq:matching-explicit-d}
P^{(d)}(\zeta)\,N(\zeta)^{-1}=I+O\!\left(\frac{1}{n}\right),
\qquad \zeta\in\partial U_d,
\end{equation}
uniformly as $n\to\infty$.
One may take
\begin{align}
\label{eq:Ec-def}
E^{(c)}(\zeta)
&:=
N(\zeta)\,
\frac{1}{\sqrt{2}}
\begin{pmatrix}
1&-i\\
-i&1
\end{pmatrix}\,
\left(n^{2/3}\xi_c(\zeta)\right)^{\sigma_3/4},\\[0.4em]
\label{eq:Ed-def}
E^{(d)}(\zeta)
&:=
N(\zeta)\,
\frac{1}{\sqrt{2}}
\begin{pmatrix}
1&-i\\
-i&1
\end{pmatrix}\,
\left(n^{2/3}\xi_d(\zeta)\right)^{\sigma_3/4},
\end{align}
with branches consistent with \eqref{eq:airy-asympt} and \eqref{eq:xi-cd}.
\end{proposition}

\begin{proof}
Substitute \eqref{eq:airy-asympt} into \eqref{eq:Pc-def}--\eqref{eq:Pd-def}.
On $\partial U_c$ and $\partial U_d$ we have $|n^{2/3}\xi_{c,d}|\asymp n^{2/3}$, hence the Airy remainder
$O(\xi^{-3/2})$ becomes $O(n^{-1})$. Choosing $E^{(c)}$ and $E^{(d)}$ as in \eqref{eq:Ec-def}--\eqref{eq:Ed-def}
cancels the constant matrix and power factors from \eqref{eq:airy-asympt}, yielding the matching relations.
\end{proof}

\begin{remark}
The matching rate $O(1/n)$ is the source of the final small-norm estimate $R=I+O(1/n)$ in
Section~\ref{sec:final-ratio}.
\end{remark}

\section{Final ratio problem and error analysis}
\label{sec:final-ratio}

\subsection{Ratio RHP for $R$}
\label{subsec:ratio-small-norm}

Let $\Sigma_R$ denote the resulting jump contour after the $g$-transformation and lens opening:
it consists of the band $\Delta=[c,d]$, the two lens lips, and the boundary circles $\partial U_c$ and
$\partial U_d$ (and, if present, the auxiliary contour $\Gamma$; cf.\ Section~\ref{sec:g-and-lens}).
Define the global piecewise parametrix $\mathcal P$ by
\begin{equation}
\label{eq:global-parametrix}
\mathcal P(\zeta):=
\begin{cases}
P^{(c)}(\zeta), & \zeta\in U_c,\\
P^{(d)}(\zeta), & \zeta\in U_d,\\
N(\zeta), & \zeta\in \C\setminus(\overline U_c\cup \overline U_d),
\end{cases}
\end{equation}
where $N$ is the outer parametrix from Section~\ref{subsec:outer} and
$P^{(c)}$, $P^{(d)}$ are the Airy parametrices from Section~\ref{subsec:airy-parametrices}.

We define the ratio function by
\begin{equation}
\label{eq:def-R}
R(\zeta):=
\begin{cases}
S(\zeta)\,(\mathcal P(\zeta))^{-1}, & \zeta\in \C\setminus \Sigma_R,
\end{cases}
\end{equation}
where $S$ denotes the unknown after the $g$-transformation and lens opening
(Section~\ref{sec:g-and-lens}). Then $R$ satisfies the following ratio RHP.

\begin{problem}[Ratio RHP for $R$]
\label{prob:R-ratio}
Find a $2\times2$ matrix $R$ such that:
\begin{itemize}[leftmargin=2.2em]
\item[(R1)] $R$ is analytic in $\C\setminus\Sigma_R$.
\item[(R2)] $R$ has jumps $R_+=R_-J_R$ on $\Sigma_R$, where
\begin{equation}
\label{eq:JR-def}
J_R(\zeta)=\mathcal P_-(\zeta)\,J_S(\zeta)\,(\mathcal P_+(\zeta))^{-1},
\qquad \zeta\in\Sigma_R,
\end{equation}
and $J_S$ denotes the jump matrix for $S$.
\item[(R3)] $R(\zeta)=I+O(\zeta^{-1})$ as $\zeta\to\infty$.
\end{itemize}
\end{problem}

\subsection{Small-norm estimate and proof of the main theorem}
\label{subsec:small-norm-proof}

\begin{lemma}[Jumps are close to $I$]
\label{lem:R-jumps}
There exists $c_0>0$ such that:
\begin{itemize}[leftmargin=2.2em]
\item[(i)] On the lens lips away from the endpoint disks,
\begin{equation}
\label{eq:JR-lips}
J_R(\zeta)=I+O\!\left(e^{-c_0 n}\right),
\qquad 
\zeta\in \Sigma_{\mathrm{lip}}^{(u)}\cup\Sigma_{\mathrm{lip}}^{(l)}\setminus(U_c\cup U_d).
\end{equation}
\item[(ii)] On the matching circles,
\begin{equation}
\label{eq:JR-circles}
J_R(\zeta)=I+O\!\left(\frac{1}{n}\right),
\qquad \zeta\in\partial U_c\cup\partial U_d.
\end{equation}
\end{itemize}
All estimates are uniform in $\zeta$ on the indicated sets.
\end{lemma}

\begin{proof}
Statement (i) follows from Lemma~\ref{lem:lip-small} and the fact that $\mathcal P$ is uniformly bounded on
compact subsets away from $\Delta$.
Statement (ii) is precisely the matching relations
\eqref{eq:matching-explicit}--\eqref{eq:matching-explicit-d}.
\end{proof}

\begin{theorem}[Small-norm solution]
\label{thm:small-norm}
The ratio RHP (Problem~\ref{prob:R-ratio}) is solvable for all sufficiently large $n$, and its solution
satisfies
\begin{equation}
\label{eq:R-small-norm}
R(\zeta)=I+O\!\left(\frac{1}{n}\right),
\qquad n\to\infty,
\end{equation}
uniformly for $\zeta\in\C\setminus\Sigma_R$.
\end{theorem}

\begin{proof}
By Lemma~\ref{lem:R-jumps}, we have $J_R-I\in L^\infty(\Sigma_R)$ with
\[
\|J_R-I\|_{L^\infty(\Sigma_R)}=O(1/n),
\]
since the exponentially small contributions are negligible compared to $1/n$.
Moreover, $J_R-I$ is supported on a fixed contour and is uniformly bounded in $L^2(\Sigma_R)$ 
as well~\cite{DeiftBook,DeiftZhou}.
Therefore, the associated singular integral operator for the ratio RHP has norm $O(1/n)$ 
on $L^2(\Sigma_R)$, so for $n$ large it is a contraction. 

Since $\Sigma_R$ is a finite union of piecewise $C^{1}$ arcs, the Cauchy singular integral operators on $\Sigma_R$
are bounded on $L^{2}(\Sigma_R)$.
Together with $J_R-I\in L^{2}(\Sigma_R)\cap L^{\infty}(\Sigma_R)$ and $\|J_R-I\|_{L^{2}\cap L^{\infty}}=O(1/n)$,
this shows that the associated Beals--Coifman operator is a contraction for $n$ large.

The standard small-norm theorem (see, e.g., \cite{DeiftBook,DeiftZhou}) 
yields existence and uniqueness of $R$, together with \eqref{eq:R-small-norm}. 
Uniformity away from $\Sigma_R$ follows from the Cauchy integral representation for $R$.
\end{proof}

\begin{remark}[Applicability of the small-norm theorem]
\label{rem:small-norm-applicability}
Here $\Sigma_R$ is a finite union of piecewise $C^{1}$ arcs (the boundaries of the endpoint disks together with
the remaining portions of the lens lips), hence it is a Carleson contour.
Moreover, the jump matrix $J_R$ satisfies $J_R-I\in L^{2}(\Sigma_R)\cap L^{\infty}(\Sigma_R)$ with
$\|J_R-I\|_{L^{2}\cap L^{\infty}}=O(1/n)$ on the circle components and $O(e^{-cn})$ on the lens lips outside
$U_c\cup U_d$.
Therefore the associated Beals--Coifman operator is well-defined and is a contraction for $n$ large, so the
standard small-norm Riemann--Hilbert theory applies (see, e.g., \cite{DeiftBook,DeiftZhou}).
\end{remark}

\begin{proof}[Proof of Theorem~\ref{thm:main}]
The estimate \eqref{eq:R-small-norm} yields $R(\zeta)=I+O(1/n)$. Reverting the explicit sequence of
transformations described in Section~\ref{subsec:outline} and using the definition of $\mathcal P$ in
\eqref{eq:global-parametrix} gives the corresponding strong asymptotics for $Y$. 
Finally, by \eqref{eq:recover-PQ} we recover the strong asymptotics 
for the multipoint Pad\'e polynomials $(P_n,Q_n)$.
\end{proof}

\end{document}